\documentclass[
a4paper,
11pt,
DIV=14,
toc=bibliography, 
headsepline, 
parskip=half-,
abstract=true,
]{scrartcl}

\usepackage[centertags]{amsmath}
\usepackage[utf8]{inputenc}
\usepackage{enumerate,amsfonts,amssymb,amsthm,amsopn,cite,array,MnSymbol,comment}

\usepackage[usenames]{color}
\definecolor{citegreen}{rgb}{0,0.6,0}
\definecolor{refred}{rgb}{0.8,0,0}
\usepackage[colorlinks, citecolor=citegreen, linkcolor=refred, urlcolor=black]{hyperref}

\title{Bubble-Tree Convergence and Local Diffeomorphism Finiteness for Gradient Ricci Shrinkers}
\author{Reto Buzano and Louis Yudowitz}
\date{}

\newtheorem{theorem}{Theorem}[section]
\newtheorem{lemma}[theorem]{Lemma}
\newtheorem{corollary}[theorem]{Corollary}
\newtheorem{proposition}[theorem]{Proposition}

\theoremstyle{definition}
\newtheorem{remark}[theorem]{Remark}

\newtheorem{definition}[theorem]{Definition}
\numberwithin{equation}{section}
\newtheorem{claim}[theorem]{Claim}

\newcolumntype{P}[1]{>{\centering\arraybackslash}p{#1}}

\newcommand{\R}{\mathbb R}
\newcommand{\eps}{\varepsilon}
\newcommand{\Lap}{\Delta}
\newcommand{\sW}{\mathcal{W}}
\newcommand{\sQ}{\mathcal{Q}}
\newcommand{\Sc}{\mathrm{R}}
\newcommand{\Ric}{\mathrm{Ric}}
\newcommand{\Rm}{\mathrm{Rm}}

\providecommand{\abs}[1]{\left| #1\right|}
\providecommand{\norm}[1]{\left|\left|#1\right|\right|}
\providecommand{\scal}[1]{\left\langle #1\right\rangle}
\providecommand{\wt}[1]{\widetilde{#1}}

\let \div \relax

\DeclareMathOperator{\Vol}{Vol}

\DeclareMathOperator{\id}{id}

\DeclareMathOperator{\div}{div}

\DeclareMathOperator*{\argmin}{arg\,min}

\newcommand\printaddress{{
\setlength{\parindent}{15pt}
\footnotesize~
\par
{\scshape Reto Buzano}
\newline 
Universit\`a degli Studi di Torino, 
Dipartimento di Matematica,
Via Carlo Alberto 10, 
10123 Torino, Italy 
\newline
\emph{E-mail address:} 
\texttt{reto.buzano@unito.it}
\par
{\scshape Louis Yudowitz}
\newline 
Queen Mary University of London, 
School of Mathematical Sciences, 
Mile End Road, 
London E1 4NS, UK
\newline
\emph{E-mail address:} 
\texttt{l.yudowitz@qmul.ac.uk}
\par
}}

\begin{document}
\maketitle
\begin{abstract}
We prove bubble-tree convergence of sequences of gradient Ricci shrinkers with uniformly bounded entropy and uniform local energy bounds, refining the compactness theory of Haslhofer--M\"uller \cite{HM11_comp1,HM15_comp2}. In particular, we show that no energy concentrates in neck regions, a result which implies a local energy identity for the sequence. Direct consequences of these results are an identity for the Euler characteristic and a local diffeomorphism finiteness theorem.
\end{abstract}


\section{Introduction and Main Results}
In this paper, we refine the compactness theory for gradient Ricci shrinkers in general dimensions. A smooth, connected, complete, $n$-dimensional Riemannian manifold $(M^n,g)$ is called a \emph{gradient Ricci shrinker} if there exists a function $f:M\to\R$, called the \emph{potential} of the shrinker, such that
\begin{equation}\label{eq.shrinker}
\Ric_g + \nabla^2_g f = \tfrac{1}{2}g.
\end{equation} 
This notion, introduced by Hamilton in \cite{Ha88_surf}, naturally generalises the concept of positive Einstein manifolds (satisfying \eqref{eq.shrinker} with $f\equiv const.$). Gradient shrinkers have been very heavily studied, particularly in the last two decades. It is not hard to see that \eqref{eq.shrinker} is equivalent to $g(t):=(1-t)\phi_t^*{g}$ satisfying Hamilton's Ricci flow equation
\begin{equation*}
\partial_t g(t) = -2\Ric_{g(t)},
\end{equation*}
where $\phi_t$ is the family of diffeomorphisms generated by $(1-t)^{-1}\nabla f$ with $\phi_0=\id_M$. That is, a gradient shrinker evolves under Ricci flow only by diffeomorphisms and scaling and becomes singular at time $t=1$. Hence, gradient shrinkers yield some of the most basic examples of singular Ricci flows.

Their importance however stems from the fact that gradient shrinkers model finite time singularities of the Ricci flow. As shown by Enders, Topping, and the first author \cite{EMT11_typeI}, for a so-called Type~I Ricci flow $(M,g(t))_{t\in [0,T)}$, a sequence of parabolic rescalings $\big(M,g_j(t):= \lambda^{-2}_j g\big(T + \lambda^2_j(t-1)\big),p\big)$ with scaling factors $\lambda_j \to 0$ will subconverge smoothly in the pointed Cheeger--Gromov sense to a gradient shrinker which is non-trivial (i.e.~non-flat) if and only if $p$ is a singular point. We also refer the interested reader to an earlier result of Naber \cite{Na10_solitons} (without the non-triviality statement) and the work of Mantegazza and the first author \cite{MM15_typeI} for an alternative proof which yields additional information about the entropy of the limiting gradient shrinker. Recently, in his spectacular trilogy \cite{Ba20_compactness, Ba20_nash, Ba20_structure}, Bamler generalised this blow-up result to the case of general Ricci flows without the Type~I assumption. Instead, one must work with a new concept of weak convergence and limiting gradient shrinkers that may have a co-dimension 4 singular set. In the special case of dimension $n=4$, Bamler has shown that one obtains \emph{orbifold} Cheeger--Gromov convergence to an \emph{orbifold Ricci shrinker} with isolated singularities modelled on $\mathbb{R}^n/\Gamma$ for some finite $\Gamma \subset O\left(n\right)$. This yields a parabolic version of the ($4$-dimensional) shrinker compactness result by Haslhofer and the first author \cite{HM11_comp1,HM15_comp2} which we will recall now.

It is now well known (see \cite{HM11_comp1}) that every gradient shrinker comes with a natural basepoint, namely a point $p := \argmin_M f$ where the potential attains its minimum. Such a point always exists and the distance between two such points is bounded by a constant depending only on the dimension. From $p$, the potential grows like one-quarter distance squared and the volume growth of geodesic balls around $p$ is at most Euclidean, see Section \ref{sec_blowup} for more details. It is therefore always possible to normalise $f$ by adding a constant so that
\begin{equation}\label{eq.normalisation}
\int_M \left(4\pi\right)^{-\frac{n}{2}}e^{-f}dV_g = 1.
\end{equation}
In this article, we always assume that the potential has been normalised this way. The gradient shrinker then has a well defined \emph{entropy},
\begin{equation*}
\mu\left(g\right)= \sW\left(g,f\right) = \int_M\left(\abs{\nabla f}_g^2 + \Sc_g + f -n\right)\left(4\pi\right)^{-\frac{n}{2}}e^{-f}dV_g>-\infty.
\end{equation*}
The entropy, introduced by Perelman in \cite{Pe02_entropy}, is non-decreasing along a general Ricci flow (in the compact case or under some technical assumptions) and assuming a lower bound for the entropy of singularity models is therefore quite natural. An additional local scalar curvature bound, which is always available for gradient shrinkers, implies no local-collapsing.

The main \emph{compactness theorem for $n$-dimensional Ricci shrinkers} from \cite{HM11_comp1} (and its improvement from \cite{HM15_comp2} that shows the condition \eqref{eq.energybounds} always holds in dimension $n=4$) then states the following.
\begin{theorem}[Theorem 1.1 in \cite{HM11_comp1} and Theorem 1.1 in \cite{HM15_comp2}] \label{hm_compactness_thm}
Let $n\geq 4$ and let $(M_i,g_i,f_i)$ be a sequence of $n$-dimensional gradient Ricci shrinkers with entropy uniformly bounded below $\mu(g_i)\geq\underline{\mu}>-\infty$. If $n>4$, then assume in addition that we have uniform local energy bounds,
\begin{equation}\label{eq.energybounds}
\int_{B_{g_i}(p_i,r)} \abs{\Rm_{g_i}}^{n/2}_{g_i} dV_{g_i} \leq E\left(r\right) < \infty, \quad \forall i,r.
\end{equation}
Then $\left(M_i,g_i,f_i,p_i\right)$ subconverges to an orbifold Ricci shrinker $(M_\infty,g_\infty,f_\infty,p_\infty)$ in the pointed orbifold Cheeger--Gromov sense where $p_i := \argmin_{M_i} f_i$. 
\end{theorem}
In particular, this means that a subsequence converges in the pointed Gromov--Hausdorff sense and in the smooth Cheeger--Gromov sense away from the isolated point singularities, see Section \ref{sec_blowup} for precise definitions of the different notions of convergence as well as for the definition of an orbifold Ricci shrinker. We denote the set of isolated singularities by $\sQ$.

This compactness result generalised earlier shrinker compactness theorems for \emph{compact} shrinkers by Cao-Sesum \cite{CS07_compactnesss}, Weber \cite{We11_compactness}, and Zhang \cite{Z06_compactness} that furthermore rely on additional conditions such as pointwise curvature bounds or positivity assumptions for the curvature. This type of orbifold compactness theorem goes back to the fundamental work on sequences of Einstein manifolds by Anderson, Bando--Kasue--Nakajima, and Tian \cite{An89_einstein, BKN89_alecoords,Na88_bubbling, Ti90_calabi}. See also Uhlenbeck \cite{U82_remove} and Cheeger--Naber \cite{CN15_regularity}.

Our aim is to further extend Theorem \ref{hm_compactness_thm} by investigating precisely what happens at the points where orbifold singularities form. In addition to the work cited above, our main results are in particular inspired by bubbling theorems for Einstein manifolds by Anderson--Cheeger \cite{AC91_diffeofin} and Bando \cite{Ba90_bubbling,Ba90_correction}. Our first main result is the following.
\begin{theorem}[Bubble Tree Convergence]\label{shrinker_bubbling}
Let $n\geq 4$, $\left(M_i, g_i, f_i, p_i\right)$ be a sequence of $n$-dimensional oriented gradient Ricci shrinkers as in Theorem \ref{hm_compactness_thm}, and $\sQ$ be the set of orbifold points of the limiting orbifold Ricci shrinker $\left(M_\infty,g_\infty,f_\infty\right)$. Then, given $q\in \sQ$, there exist point-scale sequences $\{(q^k_i, r^k_i)\}_{k=1}^{N_q}$ where $M_i \ni q^k_i \to q$, $r^k_i\to 0$, and ALE bubbles $\{(V^k,h^k,q_\infty^k)\}_{k=1}^{N_q}$ (see Definition \ref{ale}), such that, up to passing to a subsequence, the following is true.
\begin{enumerate}
\item\label{point1} For all $k\neq \ell$, we have 
\begin{equation*}
\frac{r^k_i}{r^\ell_i} + \frac{r^\ell_i}{r^k_i} + \frac{d_{g_i}(q^k_i, q^\ell_i)}{r^k_i + r^\ell_i} \to \infty
\end{equation*}
as $i\to\infty$.
\item\label{point2} For every fixed $1\leq k\leq N_q$, the pointed rescaled manifolds $(M_i, (r^k_i)^{-2}g_i,q^k_i)$ converge in the pointed orbifold Cheeger-Gromov sense to $(V^k,h^k,q_\infty^k)$ as $i\to\infty$.
\item\label{point3} Given any other sequences $M_i \ni q_i \to q$ and $\varrho_i \to 0$ such that 
\begin{equation*}
\min_{k=1,\dots N_q} \Big(\frac{\varrho_i}{r^k_i} + \frac{r^k_i}{\varrho_i} + \frac{d_{g_i}(q_i,q^k_i)}{\varrho_i + r^k_i}\Big)\to \infty
\end{equation*}
then the pointed rescaled manifolds $(M_i, (\varrho_i)^{-2}g_i,q_i)$ converge to a flat limit. 
\item\label{point4} The number of ALE bubbles forming is locally finite, in particular for every $r\geq 2$ there exists $N=N(\underline{\mu}, E(2r))$ such that $\sum_{q \in \sQ_r} N_q \leq N$, where $\sQ_r := \sQ \cap B_{g_\infty}(p_\infty,r)$. 
\item\label{point5} Finally, the following energy identity holds:
\begin{equation*}
\lim_{i \to \infty} \int_{B_{g_i}(p_i,r)} \abs{\Rm_{g_i}}^{n/2}_{g_i} dV_{g_i} = \int_{B_{g_\infty}(p_\infty,r)} \abs{\Rm_{g_\infty}}^{n/2}_{g_\infty} dV_{g_\infty} + \sum_{q \in \sQ_r} \sum_{k = 1}^{N_q} \int_{V^k}\abs{\Rm_{h^k}}^{n/2}_{h^k} dV_{h^k},
\end{equation*}
whenever $r \geq 2$ is such that $\sQ \cap \partial B_{g_\infty}(p_\infty,r) = \emptyset$.
\end{enumerate}
\end{theorem}

Bubble tree constructions as in Theorem \ref{shrinker_bubbling} are an important tool in the study of geometric PDEs and have been employed in a variety of situations. In addition to the work on Einstein manifolds cited above, we would like to mention the classical works of Sacks--Uhlenbeck \cite{SU81_bubbling} for harmonic maps, as well as the articles by Brezis--Coron \cite{BC85_bubbling} and Struwe \cite{St84_bubbling} for certain elliptic systems, all of which have inspired us. 

From a technical point of view, our proof of Theorem \ref{shrinker_bubbling} differs from the ones by Anderson--Cheeger \cite{AC91_diffeofin} and Bando \cite{Ba90_bubbling,Ba90_correction} in that we start the process from the \emph{deepest} bubble (or \emph{leaf bubble} which corresponds to the \emph{smallest} scale) and then work our way outwards, while their argument goes the other direction. This follows the first author's bubbling analysis for minimal surfaces obtained jointly with Sharp in \cite{BS18_minsurf}, as well as the beautiful work of Chang--Qing--Yang in \cite{CQY07_bubbling}.

Furthermore, working with non-compact manifolds, Theorem \ref{shrinker_bubbling} only states that the number of orbifold points in $\sQ_r = \sQ \cap B_{g_\infty}(p_\infty, r)$ is bounded by $N=N(\underline{\mu}, E(2r))$. This means that if one wishes to apply the result for large $r$, or even the entire shrinker, the number of orbifold points can become arbitrarily large and as a consequence, the bubble tree construction in Section \ref{sec_bubbling} may \emph{not} terminate.

A final important difference to the work of Anderson--Cheeger is Point \ref{point5} of Theorem \ref{shrinker_bubbling}, which is not present in \cite{AC91_diffeofin}. While proving an energy \emph{inequality} is relatively easy, proving the claimed energy \emph{identity} requires a more delicate argument to show that no energy is lost in the intermediate regions between the different bubble scales. A significant part of the present paper is therefore focusing on these so-called \emph{neck regions}, see in particular Sections \ref{sec_neck} and \ref{sec_Kato}. Once the energy identity is proved, it can easily be translated into an identity for the Euler characteristic: 

\begin{corollary}\label{cor.eulerchar}
Under the assumptions of Theorem \ref{shrinker_bubbling} and using the same notation, we have the identity
\begin{equation}\label{euler_identity}
\lim_{i \to \infty} \chi(B_{g_i}(p_i,r)) = \chi(B_{g_\infty}(p_\infty,r)\setminus \mathcal{Q}_r) + \sum_{q \in \sQ_r} \sum_{k = 1}^{N_q} \chi(V^k \setminus \mathcal{Q}^k)
\end{equation}
where $\mathcal{Q}^k$ is the (possibly empty) set of orbifold points of the ALE bubble $(V^k,h^k)$.
\end{corollary}

The concept of an ALE bubble is defined as follows.
\begin{definition}[ALE Bubble]\label{ale}
A manifold (or an orbifold with finitely many singularities) $(M^n, g)$ with one end is \emph{asymptotically locally Euclidean (ALE) of order $\tau > 0$} if there is a compact set $K \subset M^n$, a constant $R > 0$, a finite group $\Gamma \subset O\left(n\right)$ acting freely on $\mathbb{R}^n \setminus B(0,R)$, as well as a $C^\infty$ diffeomorphism $\psi: M^n \setminus K \to \left(\mathbb{R}^n \setminus B(0,R)\right)/\Gamma$ such that the following estimates hold:
\begin{align*}
(\varphi^\ast g)_{ij}(x) &= \delta_{ij} + O (\abs{x}^{-\tau})\\
\partial^k (\varphi^\ast g)_{ij}(x) &= O (\abs{x}^{-\tau-k}), \quad \forall k\geq 1
\end{align*}
for all $x,y \in \mathbb{R}^n \setminus B(0,R)$. Here $\varphi := \psi^{-1} \circ \pi$ where $\pi: \mathbb{R}^n \to \mathbb{R}^n/\Gamma$ is the natural projection. We say that an $n$-dimensional manifold (or orbifold with finitely many singularities) is an \emph{ALE bubble}, if it is complete and non-compact with one end, Ricci-flat, non-flat with bounded $L^{n/2}$ Riemannian curvature, and ALE of order $n-1$ in general. If $n = 4$ or the manifold/orbifold is K\"ahler then we require it to be ALE of order $n$. (In Definition \ref{bubble_types}, we will further distinguish between leaf and intermediate bubbles.)
\end{definition}

A further consequence of Theorem \ref{shrinker_bubbling} is a (local) diffeomorphism finiteness result for Ricci shrinkers.
\begin{corollary}[Local Diffeomorphism Finiteness]\label{shrinker_diffeo_fin}
Let $\mathcal{M}$ denote the collection of $n$-dimensional gradient Ricci shrinkers $(M,g,f)$ with entropy uniformly bounded below $\mu(g) \geq \underline{\mu} > -\infty$ and uniform local energy bounds as in \eqref{eq.energybounds} whenever $n>4$. Moreover, for any $r>0$, set $\mathcal{M}^r$ to be the collection of $M^r:=M \cap B_{g}(p,r)$, where $M\in\mathcal{M}$ and $p:= \argmin_{M} f$. Then $\mathcal{M}^r$ contains only a finite number of diffeomorphism types.
\end{corollary}

The corollary shows in particular that the collection of \emph{closed} Ricci shrinkers with a \emph{uniform upper bound on the diameter} as well as uniform lower entropy bounds and uniform energy bounds contains only a finite number of diffeomorphism types. Recently, Munteanu--Wang \cite{MW15_cone} have shown such diameter bounds follows from the other assumptions.  Proving a \emph{global} diffeomorphism finiteness result is a more delicate issue; even in the case where one knows that all orbifold singularities form in a compact region (which is for example the case under a scalar curvature bound), one still needs to control the number of ends of the shrinkers, a problem which we will study elsewhere.

The paper is organised as follows: In Section \ref{sec_blowup}, we recall some basic concepts and collect some facts about gradient Ricci shrinkers before proving a blow-up version of Theorem \ref{hm_compactness_thm} (see Theorem \ref{blow-up.thm}). In Section \ref{sec_neck}, we first show a connectedness result for small annuli in Ricci shrinkers (Lemma \ref{small_annuli}) which implies that bubbles have precisely one end (Corollary \ref{cor.oneend}), and then we proceed to prove a neck theorem (Theorem \ref{shrinker_neck_thm}) controlling the geometry of intermediate regions in our bubbling result. Section \ref{sec_Kato} is dedicated to proving an energy estimate in these neck regions (Theorem \ref{no_neck_energy}) via an improved Kato inequality for Ricci shrinkers. Theorem \ref{shrinker_bubbling} is then proved in Section \ref{sec_bubbling}. In Section \ref{sec_further}, we prove the two corollaries from above.

\paragraph{Acknowledgements.} RB has been partially supported by the EPSRC grant EP/S012907/1 and LY has been supported by a studentship from the QMUL Faculty of Science and Engineering Research Support Fund. We also thank an anonymous referee for some valuable comments.

\section{A Blow-up Version of the Compactness Theorem}\label{sec_blowup}
Let us start this section with the precise notions of pointed Gromov--Hausdorff convergence and pointed orbifold Cheeger--Gromov convergence and a quick overview of the main results from \cite{HM11_comp1,HM15_comp2}.
\begin{definition}[Pointed Gromov--Hausdorff Convergence]\label{GHA}
A pointed map $f: (X,p) \to (Y,q)$ between two metric spaces $\left(X, d_X, p\right)$, $\left(Y,d_Y,q\right)$ is an \emph{$\eps$-pointed Gromov--Hausdorff approximation} ($\eps$-PGHA) if it is almost an isometry and almost onto in the following sense
\begin{enumerate}[i)]
\item $\abs{d_X(x_1, x_2) - d_Y(f(x_1), f(x_2))} \leq \eps$, for all $x_1,x_2 \in B_{d_X}(p,1/\eps)$,
\item for all $y \in B_{d_Y}(q,1/\eps)$ there exists $x\in B_{d_X}(p,1/\eps)$ with $d_Y(y,f(x))\leq \eps$.
\end{enumerate}
We say $\left(X_i,p_i\right) \to \left(Y,q\right)$ as $i\to\infty$ in the \emph{pointed Gromov--Hausdorff} sense if 
\begin{align*}
d_{\mathrm{pGH}}((X_i,p_i), (Y,q)) &:= \inf \{\eps > 0 : \exists \text{ $\eps$-pGHA } f_1:(X_i,p_i) \to (Y,q) \text{ and } f_2:(Y,q) \to (X_i,p_i)\}\\ 
&\to 0 \qquad (i\to\infty).
\end{align*}
\end{definition}
As explained in \cite{HM11_comp1}, Lemma 2.1 and Lemma 2.2, under the normalisation \eqref{eq.normalisation}, one obtains the following growth condition for the potential $f$ from the basepoint $p$,
\begin{equation*}
\frac{1}{4}(d(x,p)-5n)_{+}^2 \leq f(x) - \mu(g) \leq \frac{1}{4}(d(x,p)+\sqrt{2n})^2.
\end{equation*}
This in turn implies the volume growth estimate
\begin{equation}\label{eq.volumegrowth}
\Vol_g(B_g(p,r)) \leq V_0 r^n, \quad \forall r>0
\end{equation}
with $V_0$ being a constant depending only on the dimension of the shrinker. Finally, similar to Perelman's non-collapsing result, under the entropy bound $\mu(g)\geq \underline{\mu}>-\infty$ one obtains for every $r$ the existence of $v_0 =v_0(r,n,\underline{\mu})$ such that
\begin{equation}\label{eq.noncollapsing}
\Vol_g(B_g(q,\delta)) \geq v_0 \delta^n,
\end{equation}
for every ball $B_g(q,\delta)\subset B_g(p,r)$, $0<\delta\leq 1$, see Lemma 2.3 in \cite{HM11_comp1}. Pointed Gromov--Hausdorff convergence of a sequence of Ricci shrinkers $(M_i,g_i,f_i,p_i)$ with entropy $\mu(g_i)\geq \underline{\mu}>-\infty$ to a complete metric space $(M_\infty,d_\infty,p_\infty)$ then follows directly from \eqref{eq.volumegrowth}--\eqref{eq.noncollapsing} and Gromov's compactness theorem, see Theorem 2.4 in \cite{HM11_comp1} for details.

The main work of \cite{HM11_comp1,HM15_comp2} then goes into improving the regularity of the convergence and of the limit metric space $M_\infty$. 
\begin{definition}[Orbifold Ricci Shrinker]
A complete metric space $M_\infty$ is called an \emph{orbifold Ricci shrinker} if it is a smooth Ricci shrinker away from a locally finite set $\sQ$ of singular points and at every $q\in\sQ$, $M_\infty$ is modelled on $\mathbb{R}^n/\Gamma$ for a finite group $\Gamma \subset O(n)$. Moreover, there exists an associated covering $\mathbb{R}^n \supset B(0,\varrho)\setminus \{0\} \stackrel{\pi}{\to} U \setminus \{q\}$ of some neighbourhood $U\subset M_\infty$ of $q$ such that $(\pi^*g_\infty,\pi^*f_\infty)$ can be extended smoothly to a gradient shrinker over the origin.
\end{definition}
\begin{definition}[Pointed Orbifold Cheeger--Gromov Convergence]
A sequence of gradient shrinkers $\left(M^n_i, g_i, f_i, p_i\right)$ converges to an orbifold gradient shrinker $\left(M^n_\infty, g_\infty, f_\infty, p_\infty\right)$ in the \emph{pointed orbifold Cheeger--Gromov sense} if the following properties hold:
\begin{enumerate}
\item There exist a locally finite set $\mathcal{Q} \subset M_\infty$, an exhaustion of $M_\infty \setminus \mathcal{Q}$ by open sets $U_i$, and smooth embeddings $\varphi_i: U_i \to M_i$ such that $\left(\varphi^\ast_i g_i, \varphi^\ast_i f_i\right)$ converges to $\left(g_\infty, f_\infty\right)$ in the $C^\infty_{\mathrm{loc}}$-sense on $M_\infty \setminus \mathcal{Q}$.
\item Each of the above maps $\varphi_i$ can be extended to an $\eps$-pGHA which yield a convergent sequence $\left(M_i, d_i, p_i\right) \to \left(M_\infty, d_\infty, p_\infty\right)$ in the pointed Gromov--Hausdorff sense.
\end{enumerate}
Pointed orbifold Cheeger--Gromov convergence to a Ricci-flat orbifold is defined analogously.
\end{definition}
The main result of \cite{HM11_comp1} improves the Gromov--Hausdorff convergence to orbifold Cheeger--Gromov convergence under the energy bound \eqref{eq.energybounds} and the main result of \cite{HM15_comp2} shows that the energy bound assumption is in fact always satisfied in dimension $n=4$, see Theorem \ref{hm_compactness_thm} from the introduction. It turns out that the set $\sQ$ is the same in the two above definitions, meaning that the convergence is bad (or non-smooth) around a point $q\in M_\infty$ if and only if $M_\infty$ is non-smooth at $q$. We can therefore use the expressions that $q$ is a singular point or a point of bad convergence interchangeably. A key ingredient in the proof of this improved convergence result is an $\eps$-regularity theorem that follows from local Sobolev constant bounds via a Moser iteration argument. We recall these two results here as we will need them later.
\begin{lemma}[Local Sobolev constant bounds, Lemma 3.2 in \cite{HM11_comp1}]\label{Sobolev_bounds}
There exist $C_S(r)<\infty$ and $\delta_0(r)>0$ depending on $r$, $n$ and $\underline{\mu}$, such that for every gradient shrinker with normalised weighted volume and $\mu(g) \geq \underline{\mu} > -\infty$, and for every ball $B_g(x,\delta) \subset B_g(p,r)$ with $0 < \delta \leq \delta_0(r)$ and $p := \argmin_M f$, we have
\begin{equation*}
\norm{\varphi}_{L^{2^*}} < C_S(r)\norm{\nabla \varphi}_{L^2},
\end{equation*}
for all functions $\varphi\in C^1_c(B_g(x,\delta))$, where $2^*=\tfrac{2n}{n-2}$.
\end{lemma}
\begin{lemma}[$\eps$-Regularity, Lemma 3.3 in \cite{HM11_comp1}]\label{shrinker_eps_reg}
There exist $\eps_{\mathrm{reg}}(r), \delta_0(r)>0$, and $K_\ell(r) < \infty$, all depending on $r$, $n$ and $\underline{\mu}$, such that for every gradient shrinker with normalised weighted volume and $\mu(g) \geq \underline{\mu} > -\infty$, and for every ball $B_g(x,\delta) \subset B_g(p,r)$ with $0 < \delta \leq \delta_0(r)$ and $p := \argmin_M f$, we have the implication
\begin{equation*}
\norm{\Rm_g}_{L^{n/2}(B_g(x,\delta))} < \eps_{\mathrm{reg}}(r) \Longrightarrow \sup \limits_{B_g(x,\delta/4)} \abs{\nabla^\ell \Rm_g}_g \leq \frac{K_\ell(r)}{\delta^{2+\ell}} \norm{\Rm_g}_{L^{n/2}(B_g(x,\delta))}.
\end{equation*}
\end{lemma}
Under the energy bound \eqref{eq.energybounds}, for a large $r$ and a small $\delta>0$, there can only be finitely many disjoint $\delta$-balls in $B_g(p,r)$ that contain energy more than $\eps_{\mathrm{reg}}(r)$. In light of Lemma \ref{shrinker_eps_reg}, away from these balls we get $C^\infty$ estimates for the curvatures and hence smooth convergence. Hence the singular points $q\in \sQ$ are exactly characterised by the condition
\begin{equation}\label{eq.energyconcentration}
\exists q_i\in M_i \text{ with } q_i\to q, \exists \delta_i\to 0 \text{ such that }\norm{\Rm_{g_i}}_{L^{n/2}(B_{g_i}(q_i,\delta_i))} \geq \eps_{\mathrm{reg}}(r).
\end{equation}
Our first new result is a blow-up version of Theorem \ref{hm_compactness_thm} stating that if we rescale the metrics of our sequence of Ricci shrinkers with $\lambda_i^{-2}$ (where $\lambda_i\to 0$) we still obtain orbifold Cheeger--Gromov convergence. In order to allow us to apply this result flexibly in different situations below, we prove a rather general theorem which does not yet make a statement about whether or not the limit is flat and whether or not it has singular points -- properties that will in particular depend on the precise choice of $q$ and the scaling factors $\lambda_i$.
\begin{theorem}[Blow-up Version of Theorem \ref{hm_compactness_thm}]\label{blow-up.thm}
Let $\left(M_i, g_i, f_i, p_i\right)$ be a sequence of $n$-dimensional gradient Ricci shrinkers with uniformly bounded entropy $\mu\left(g_i\right) \geq \underline{\mu} > -\infty$ and, if $n > 4$, locally bounded energy as in Theorem \ref{hm_compactness_thm}. Let $q\in M_\infty$ and let $M_i \ni q_i \to q$ and $\lambda_i \to 0$. Then the rescaled sequence $(M_i, \wt{g}_i = \lambda_i^{-2}g_i, q_i)$ subconverges in the pointed (orbifold) Cheeger--Gromov sense to a complete, non-compact, Ricci-flat manifold or orbifold with isolated singularities $(V, h, q_\infty)$ which has bounded $L^{n/2}$ Riemannian curvature and is ALE of order $n-1$ in general and ALE of order $n$ if either $n = 4$ or $(V,h)$ is K\"ahler. Finally, the singular points of $V$ are characterised by \eqref{eq.energyconcentration} for the rescaled metrics $\wt{g}_i$.
\end{theorem}

\begin{proof}
The proof consists of checking that after rescaling we can essentially still follow the same arguments as in the original proof of Theorem \ref{hm_compactness_thm} to obtain orbifold Cheeger--Gromov convergence and in checking the claimed properties of the limiting manifold or orbifold.

First, choose some $r \geq 2$ such that $q\in B_{g_\infty}(p_\infty, r)$. By picking $i$ sufficiently large, we may assume that $q_i \in B_{g_i}(p_i,r+1)$ and thus $B_{g_i}(q_i,1) \subset B_{g_i}(p_i,r+2) \subset B_{g_i}(p_i,2r)$, which by \eqref{eq.volumegrowth} implies that $\Vol_{g_i} B_{g_i}(q_i,1) \leq C_n(2r)^n$ independently of $i$. Clearly, these unit balls $B_{g_i}(q_i,1)$ with respect to the original metrics correspond to the larger and larger balls $B_{\wt{g}_i}(q_i, \lambda_i^{-1})$ with respect to the rescaled metrics. Using also \eqref{eq.noncollapsing}, we therefore see that there are constants $v_1, V_1$ depending only on $r$, $n$ and $\underline{\mu}$, such that
\begin{equation}\label{eq.rescaledvolumebounds}
v_1 s^n \leq \Vol_{\wt{g}_i} B_{\wt{g}_i}(q_i,s) \leq V_1 s^n
\end{equation}
whenever $i$ is sufficiently large so that $s<\lambda_i^{-1}$. In particular, this controls the number of small balls that can be placed disjointly in a large ball and thus implies pointed Gromov--Hausdorff convergence to a complete length space by Gromov's compactness theorem. Clearly this limit space is non-compact.

Next, we note that Lemma \ref{shrinker_eps_reg} still holds for the rescaled metrics $\wt{g}_i$ and for balls $B_{\wt{g}_i}(x,\delta)$ such that $0<\lambda_i \delta \leq \delta_0(r)$. This is obtained by scaling $\wt{g}_i$, applying the $\eps$-regularity lemma for shrinkers, and then scaling back. More precisely, we have
\begin{equation}\label{eq.rescaledepsreg}
\begin{aligned}
\norm{\Rm_{\wt{g}_i}}_{L^{n/2}(B_{\wt{g}_i}(x,\delta))} < \eps_{\mathrm{reg}}(r) &\Longleftrightarrow \norm{\Rm_{g_i}}_{L^{n/2}(B_{g_i}(x,\lambda_i\delta))} < \eps_{\mathrm{reg}}(r)\\
& \Longrightarrow \sup \limits_{B_{g_i}(x,\lambda_i\delta/4)} \abs{\nabla^\ell \Rm_{g_i}}_{g_i} \leq \frac{K_\ell(r)}{(\lambda_i\delta)^{2+\ell}} \norm{\Rm_{g_i}}_{L^{n/2}(B_{g_i}(x,\lambda_i\delta))}\\
&\Longleftrightarrow \sup \limits_{B_{\wt{g}_i}(x,\delta/4)} \abs{\nabla^\ell \Rm_{\wt{g}_i}}_{\wt{g}_i} \leq \frac{K_\ell(r)}{\delta^{2+\ell}} \norm{\Rm_{\wt{g}_i}}_{L^{n/2}(B_{\wt{g}_i}(x,\delta))}.
\end{aligned}
\end{equation}
So for the rescaled metrics we have the exact same implication as in Lemma \ref{shrinker_eps_reg}, the \emph{advantage} being that we can potentially work with much larger balls, a fact that we will use in the neck theorem below to conclude \emph{flatness} of the limit.

Endowed with such an $\eps$-regularity result, we can conclude exactly as in \cite{HM11_comp1} to improve the regularity of the limit to an orbifold $(V,h)$ with isolated singularities and the convergence to pointed orbifold Cheeger--Gromov convergence. We refer the reader to Section $3$ of \cite{HM11_comp1} and the associated references for more details. In the exact same way as described above, the orbifold points of $V$ are exactly the points where the convergence is bad and these points are characterised by an energy concentration as in \eqref{eq.energyconcentration} for the rescaled metrics $\wt{g}_i$. As the bounded $L^{n/2}$ Riemannian curvature of the limit $(V,h)$ is an obvious consequence of the local energy bound \eqref{eq.energybounds}, it only remains to show that the limit is Ricci-flat and satisfies the ALE condition.

To prove the former property, note that the rescaling changes \eqref{eq.shrinker} to
\begin{equation*}
\Ric_{\wt{g}_i} + \nabla^2_{\wt{g}_i} f_i = \frac{\lambda^2_i}{2}\wt{g}_i.
\end{equation*}
Hence, away from the points of bad convergence, $(V,h)$ satisfies the steady soliton equation
\begin{equation*}
\Ric_h + \nabla^2 f = 0
\end{equation*}
for some function $f:V\to \mathbb{R}$. Since any Ricci shrinker $(M_i,g_i,f_i,p_i)$ satisfies
\begin{equation*}
0 \leq R_{g_i}(x) \leq f_i(x) - \mu(g_i) \leq \frac{1}{4}(d_{g_i}(x,p_i)+\sqrt{2n})^2,
\end{equation*}
see for example (2.11) in \cite{HM11_comp1}, and the rescaled metrics satisfy $R_{\wt{g}_i} = \lambda^2_i R_{g_i}$, we also conclude that $(V,h)$ is scalar-flat. If $(V,h)$ is a smooth manifold (and thus a smooth steady soliton), then it satisfies
\begin{equation*}
R_h = \Lap R_h + 2 \abs{\Ric_h}^2
\end{equation*}
and we can directly conclude Ricci-flatness from scalar-flatness. This argument does not directly go through if there are orbifold singularities, but we can work instead with the evolution equation for the scalar curvature on the shrinkers $(M_i,g_i,f_i)$, namely
\begin{equation*}
R_{g_i} + \scal{\nabla f_i, \nabla R_{g_i}} = \Lap R_{g_i} + 2\abs{\Ric_{g_i}}^2_{g_i}
\end{equation*}
and pass to a limit after rescaling to conclude that the limit is Ricci-flat.

In the final step, we want to apply the following theorem which will yield the desired ALE condition.
\begin{theorem}[Theorem 1.5 in \cite{BKN89_alecoords}]\label{einstein_ale}
Let $(V^n, h)$ with $n \geq 4$ be a Ricci-flat manifold or a Ricci-flat orbifold with isolated singularities such that for some $x\in M$ and $v>0$, we have
\begin{equation*}
\Vol_h(B_h(x,s)) \geq vs^n, \qquad \forall s>0
\end{equation*}
as well as
\begin{equation*}
\int_V \abs{\Rm_h}^{n/2}_h dV_h \leq C < \infty.
\end{equation*}
Then $(V^n, h)$ is ALE of order $n-1$. If $n = 4$ or $(V^n, h)$ is K\"ahler then it is ALE of order $n$.
\end{theorem}
In order to apply this theorem, we pick $x\in V$ to be the limit $q_\infty$ of the points $q_i$ (whether this is a point of good or bad convergence does not matter). Then note that the volume growth assumption follows by passing to a limit in \eqref{eq.rescaledvolumebounds} while the integral condition follows by
\begin{align*}
\int_V \abs{\Rm_h}^{n/2}_h dV_h &= \lim_{s\to\infty}\int_{B_h(q_\infty,s)} \abs{\Rm_h}^{n/2}_h dV_h \leq \lim_{s\to\infty} \liminf_{i\to\infty} \int_{B_{\wt{g}_i(q_i,s)}} \abs{\Rm_{\wt{g}_i}}^{n/2}_{\wt{g}_i} dV_{\wt{g}_i}\\
&= \lim_{s\to\infty} \liminf_{i\to\infty} \int_{B_{g_i(q_i,\lambda_i s)}} \abs{\Rm_{g_i}}^{n/2}_{g_i} dV_{g_i} \leq \int_{B_{g_i(q_i,1)}} \abs{\Rm_{g_i}}^{n/2}_{g_i} dV_{g_i}\\ 
&\leq \int_{B_{g_i(p_i,2r)}} \abs{\Rm_{g_i}}^{n/2}_{g_i} dV_{g_i} \leq E(2r) < \infty.
\end{align*}
On the second line, we used that for any $s$ we may take $i$ large enough so that $s<\lambda_i^{-1}$ and on the last line we used the uniform local energy bound \eqref{eq.energybounds} which we assumed for $n>4$ and which, as mentioned previously, is always automatically satisfied for $n=4$ by the work in \cite{HM15_comp2}. This completes the proof of Theorem \ref{blow-up.thm}.
\end{proof}

We also recall the following sufficient condition for flatness of the limit.
\begin{proposition}[Bando's gap result, \cite{Ba90_bubbling,Ba90_correction}]\label{prop.Bandogap}
There exists $\eps_{\mathrm{gap}}(n)>0$ such that the following holds. Let $(V^n,h)$ be a limit orbifold arising in Theorem \ref{blow-up.thm} with
\begin{equation}\label{eq.Vflat}
\int_V \abs{\Rm_h}^{n/2}_h dV_h < \eps_{\mathrm{gap}}.
\end{equation}
Then $(V,h)$ is flat, i.e. $\Rm_h \equiv 0$ in the regular part of $V$.
\end{proposition}

\section{Small Annuli in Ricci Shrinkers and a Neck Theorem}\label{sec_neck}
Let $\bar{A}_{s_1,s_2}(q)$ denote the closed \emph{geodesic annulus} centered at some point $q \in M$. That is,
\begin{equation*}
\bar{A}_{s_1,s_2}(q) := \overline{B_g(q, s_2)} \setminus B_g(q, s_1).
\end{equation*}
Furthermore, let $A_{s_1,s_2}(q)$ denote a \emph{connected component} of $\bar{A}_{s_1,s_2}(q)$ such that
\begin{equation}\label{eq.compsphere}
A_{s_1, s_2}(q) \cap \partial B_g(q, s_2) \neq \emptyset.
\end{equation}
The first lemma of this section, which is important for the neck theorem below, shows that for sufficiently small annuli $\bar{A}_{s_1,s_2}(q)$ there exists \emph{only one} such component $A_{s_1,s_2}(q)$. We will also prove a diameter bound for this component. These results are Ricci shrinker versions of results from \cite{AG90_diamgrowth, AC91_diffeofin} for manifolds with pointwise Ricci bounds. We only consider annuli lying fully inside $B_g(p, 2r)$ for some fixed $r\geq 2$ and, in order to guarantee that we can later apply our results to our sequence of shrinkers, we make sure that all constants only depend on this $r$ as well as possibly on $n$ and $\underline{\mu}$.

The key tool to prove these results is the Bakry--\'Emery volume comparison theorem (Theorem $1.2$ in \cite{WW09_comparison}) which implies that for a gradient Ricci shrinker there exists $a=a(r,n)$ such that 
\begin{equation}\label{eq.WW09_comparison}
\frac{\Vol_f(B_g(q, \sigma_2))}{\Vol_f(B_g(q, \sigma_1))} \leq a(r,n)\,\frac{\sigma_2^n}{\sigma_1^n}.
\end{equation}
whenever $0<\sigma_1\leq \sigma_2 \leq 1$ and $q \in B_g(p,r+1)$. Here, $p := \argmin_M f$ is the basepoint of the shrinker, $r\geq 2$, and $\Vol_f (\Omega) := \int_\Omega e^{-f}dV$. The constant $a$ can in fact simply be chosen to be $a=e^b$ where $b$ is a bound on $\abs{\nabla f}$ on the ball $B_g(p, 2r)$. Such a bound follows from the auxiliary equation for Ricci shrinkers
\begin{equation*}
R_g + \abs{\nabla f}^2_g - f = -\mu(g).
\end{equation*}
which, together with $R_g \geq 0$ yields
\begin{equation}\label{eq.nablafest}
0 \leq \abs{\nabla f}^2_g \leq R_g + \abs{\nabla f}^2_g = f - \mu(g) \leq \frac{1}{4}(d(x,p)+\sqrt{2n})^2.
\end{equation}
So we can for example pick $a(r,n)=\exp(r+\sqrt{n})$. There is also a corresponding version of \eqref{eq.WW09_comparison} for annuli, obtained in the proof of Theorem 1.2 in \cite{WW09_comparison}. The version we will use below can be written as
\begin{equation}\label{eq.ww_annuli}
\frac{\Vol_f(\bar{A}_{\sigma_0,\sigma_2}(q))}{\Vol_f(\bar{A}_{\sigma_0,\sigma_1}(q))} \leq a(r,n)\,\frac{\sigma_2^n-\sigma_0^n}{\sigma_1^n-\sigma_0^n},
\end{equation}
whenever $0<\sigma_0<\sigma_1\leq \sigma_2 \leq 1$ and $q \in B_g(p,r+1)$.
\begin{lemma}[Small Annuli in Ricci Shrinkers]\label{small_annuli}
Given $n\geq 4$, $r\geq 2$ and $\underline{\mu}>-\infty$, then there exist constants $0<\zeta_0<\tfrac{1}{3}$, and $C_0<\infty$, such that the following holds: Let $(M,g,f)$ be a complete, connected $n$-dimensional gradient Ricci shrinker with entropy bounded below by $\mu(g) \geq \underline{\mu}$. Let $p := \argmin_M f$ and $q \in B_g(p, r+1)$. If
\begin{equation*}
s_2 \leq \tfrac{1}{4} \quad\text{and}\quad s_1 = \zeta s_2 \quad\text{where}\quad \zeta\leq \zeta_0
\end{equation*}
then the annulus $\bar{A}_{s_1,s_2}(q)$ contains at most one connected component that meets $\partial B_g(q, s_2)$ in the sense of \eqref{eq.compsphere} and, if such a component exists, any two points in $A_{s_1, s_2}(q) \cap \partial B_g(q, s_2)$ can be connected by a curve lying in $A_{s_1, 2s_2}(q)$ of length $C_0s_2$.
\end{lemma}
\begin{remark}
Obviously any two points in $A_{s_1, s_2}(q) \cap \partial B_g(q, s_2)$ can be connected in $M$ by a curve of length $2s_2$ (by passing via $q$), but the lemma excludes curves getting near $q$. By a more careful argument, one can prove an actual \emph{intrinsic} diameter bound for $\bar{A}_{s_1, s_2}(q)$, but the claimed property has a simple proof and is sufficient for us.
\end{remark}

\begin{proof}
Assume for a contradiction that there are at least two such components $\{D_i\}$. We may assume that $D_1$ is such that for any $i \neq 1$ we have $\Vol_f(D_1) \leq \Vol_f (D_i)$ which implies 
\begin{equation*}
\Vol_f(\bar{A}_{s_1, s_2}(q)) \leq 2\Vol_f\Big(\bigcup_{i\neq 1} D_i\Big).
\end{equation*}
Now choose $x \in D_1 \cap \partial B_g(q,s_2)$. Note that any minimising geodesic $\gamma(t)$ in $M$ from $x$ to some component $D_i$ (with $i\neq 1$) has length at most $2s_2$ and intersects $B_g(q,s_1)$ for some $t_0 \in [s_2 - s_1, s_2 + s_1]$. Let $A^0_{u,v}(x)$ be the set of all points on such a geodesic with $t \in [u,v]$. Note the inclusion relation
\begin{equation}
\bigcup_{i\neq 1} D_i \subseteq A^0_{s_2 - s_1, 2s_2}(x).
\end{equation}
Also, the triangle inequality together with $\gamma(t_0)\subset B_g(q,s_1)$ yields
\begin{equation}\label{eq.geodesicann}
A^0_{s_2 - s_1, s_2+s_1}(x) \subseteq B_g(q, 3s_1).
\end{equation}	
Combining \eqref{eq.ww_annuli}--\eqref{eq.geodesicann}, we obtain
\begin{align*}
\frac{\Vol_f(\bar{A}_{s_1, s_2}(q))}{\Vol_f(B_g(q, 3s_1))} &\leq \frac{2\Vol_f\Big(\bigcup_{i\neq 1} D_i\Big)}{\Vol_f(B_g(q, 3s_1))}\\ 
&\leq \frac{2\Vol_f(A^0_{s_2 - s_1, 2s_2}(x))}{\Vol_f(A^0_{s_2 - s_1, s_2+s_1}(x))}\\ 
&\leq 2a(n,r)\frac{(2\zeta^{-1})^n -(\zeta^{-1}-1)^n}{(\zeta^{-1}+1)^n -(\zeta^{-1}-1)^n}\\ 
&\leq C_0 \zeta^{-1} + C_1
\end{align*}
for constants $C_0$ and $C_1$ depending only on $n$ and $a(r,n)$. On the other hand, we also have
\begin{equation*}
\frac{\Vol_f(\bar{A}_{s_1, s_2}(q))}{\Vol_f(B_g(q, 3s_1))} \geq \frac{\Vol_f(B_g(q, s_2))}{\Vol_f(B_g(q, 3s_1))} - 1 \geq \frac{v_0 s_2^n}{3^n V_0 s_1^n}-1 = \frac{v_0}{3^nV_0} \zeta^{-n}-1,
\end{equation*}
where $V_0$ and $v_0$ are the constants from \eqref{eq.volumegrowth}--\eqref{eq.noncollapsing}. Clearly this yields a contradiction when $\zeta_0$ is sufficiently small (and hence $\zeta^{-1}\geq \zeta_0^{-1}$ sufficiently large), showing that there can be at most one connected component $A_{s_1,s_2}(q)$ that meets $\partial B_g(q, s_2)$.

It remains to prove the claimed diameter bound. In order to do so, pick a maximal family of points $x_j \in A_{s_1, s_2}(q) \cap \partial B_g(q, s_2)$ such that $B_j := B_g(x_j, \xi s_2)$ are disjoint for $\xi:=\tfrac{1}{2}(1-\zeta)$ and set $\hat{B}_j :=B_g(x_j, 2\xi s_2)$. Clearly if $\hat{B}_j \cap \hat{B}_k \neq \emptyset$, then $x_j$ and $x_k$ can be joined by a curve in $\bar{A}_{s_1, 2s_2}(q)$ of length at most $4\xi s_2$. This uses in particular that all $\hat{B}_j$ are disjoint from $B_g(q,s_1)$ by definition of $\xi$. By maximality, $\{\hat{B}_j\}$ cover $A_{s_1, s_2}(q) \cap \partial B_g(q, s_2)$ and therefore any two points in $A_{s_1, s_2}(q) \cap \partial B_g(q, s_2)$ can be joined by a curve in $\bar{A}_{s_1, 2s_2}(q)$ of length as most $4\xi s_2 \cdot \#\{x_j\}$. It remains to estimate the number of points in the family $\{x_j\}$.

Note the inclusion
\begin{equation*}
B_j = B_g(x_j, \xi s_2) \subseteq B_g(q, (1 + \xi)s_2) \subseteq B_g(x_j, (2 + \xi)s_2).
\end{equation*}
which by \eqref{eq.WW09_comparison} yields
\begin{equation*}
\frac{\Vol_f (B_g(q,(1+\xi)s_2))}{\Vol_f B_j} \leq \frac{\Vol_f (B_g(x_j,(2+\xi)s_2))}{\Vol_f B_j} \leq a(n,r)\left(\frac{2 + \xi}{\xi}\right)^n
\end{equation*}
for each $j$. In particular, the number of disjoint $B_j$ lying in $B_g(q, (1 + \xi)s_2)$ is bounded by $a(n,r)(\tfrac{2 + \xi}{\xi})^n$.

Combining with the above, we see that any two points in $A_{s_1, s_2}(q) \cap \partial B_g(q, s_2)$ can be joined by a curve in $\bar{A}_{s_1, 2s_2}(q)$ of length as most $4\xi a(n,r)(\tfrac{2 + \xi}{\xi})^n \cdot s_2$. An explicit constant $C_0$ can easily be obtained from $0<\zeta<\tfrac{1}{3}$ which yields $\tfrac{1}{3}<\xi<\tfrac{1}{2}$.
\end{proof}

\begin{corollary}[One End]\label{cor.oneend}
Any limit manifold or orbifold $(V,h)$ obtained in the blow-up version of the compactness theorem, Theorem \ref{blow-up.thm}, has one end.
\end{corollary}

\begin{proof}
For sufficiently small $\lambda_i$, the annuli $\bar{A}_{\lambda_i,\sqrt{\lambda_i}}(q_i)$ satisfy the assumptions of Lemma \ref{small_annuli} and therefore have only one connected component meeting the outer boundary. This immediately implies that $(V,h)$ has only one end.
\end{proof}

In the remainder of this section, we prove a so-called neck theorem. Generally, in bubbling results one usually encounters three different types of regions: regions where energy concentrates (and bubbles form), regions where there is no such concentration (and the convergence is smooth), and finally the intermediate or neck regions. The following result about these intermediate regions is a Ricci shrinker version of Theorem 1.8 in \cite{AC91_diffeofin} for manifolds with pointwise Ricci bounds.
\begin{theorem}[Neck Theorem for Ricci Shrinkers]\label{shrinker_neck_thm}
Let $n\geq 4$, $r\geq 2$, $\underline{\mu} > -\infty$, $k\in\mathbb{N}$ and $\eps>0$ be given constants. Then there exist $\eps_{\mathrm{neck}}>0$, $\sigma_1>0$ and $\gamma<\infty$ such that the following holds.

Let $(M, g, f)$ be a complete $n$-dimensional gradient Ricci shrinker such that $\mu(g) \geq \underline{\mu}$ and the local energy bounds \eqref{eq.energybounds} are satisfied if $n>4$. Take $q \in B_g(p,r+1)$ where $p := \argmin_M f$. Let $A_{s_1, s_2}(q) \subset M$ be the unique connected component of the geodesic annulus $\bar{A}_{s_1, s_2}(q)$ which satisfies the condition $A_{s_1, s_2}(q) \cap \partial B_g(q,s_2) \neq \emptyset$ (according to Lemma \ref{small_annuli}) and with
\begin{equation}\label{eq.neckcondition1}
s_2 \leq \sigma_1, \qquad s_1 \leq \eps_{\mathrm{neck}} s_2.
\end{equation}
Finally, assume that
\begin{equation}\label{eq.neckcondition2}
\int_{A_{s_1, s_2}(q)} \abs{\Rm_g}^{n/2}_g dV_g \leq \eps_{\mathrm{neck}}.
\end{equation}
Then there is some $\Gamma \subset O(n)$ acting freely on $S^{n-1}$ with $\abs{\Gamma} \leq \gamma$ and an $\eps$-quasi-isometry\footnote{We call a map $\psi:X\to Y$ between metric spaces an $\eps$-quasi-isometry if $\abs{d_X(x_1, x_2) - d_Y(\psi(x_1), \psi(x_2))} \leq \eps$, for all $x_1,x_2 \in X$.} $\psi$ with
\begin{equation}\label{eq.closetocone}
A_{(\eps_{\mathrm{neck}}^{-1/2} + \eps)s_1, (\eps_{\mathrm{neck}}^{1/2}-\eps)s_2}(q) \subset \psi \Big(\mathcal{C}_{\eps_{\mathrm{neck}}^{-1/2}s_1, \eps_{\mathrm{neck}}^{1/2}s_2}(S^{n-1}/\Gamma)\Big) \subset A_{(\eps_{\mathrm{neck}}^{-1/2} - \eps)s_1, (\eps_{\mathrm{neck}}^{1/2}+\eps)s_2}(q)
\end{equation}
such that for all $\mathcal{C}_{\frac{1}{2}s,s} (S^{n-1}/\Gamma) \subset \mathcal{C}_{\eps_{\mathrm{neck}}^{-1/2}s_1, \eps_{\mathrm{neck}}^{1/2}s_2}(S^{n-1}/\Gamma)$ in local coordinates one has
\begin{equation}\label{eq.epsclosetocone}
\abs{(\psi^\ast (s^{-2}g))_{ij} - \delta_{ij}}_{C^k} \leq \eps.
\end{equation}
\end{theorem}

\begin{proof}
The proof is in two steps. We first prove the following claim:
\begin{claim}
There exist $\eps_{\mathrm{neck}}, \sigma_1, \gamma$ such that for each $s$ as in the statement of the theorem, \eqref{eq.closetocone}--\eqref{eq.epsclosetocone} hold for some $\psi_s: \mathcal{C}_{\frac{1}{2}s,s}(S^{n-1} / \Gamma_s) \to A_{\eps_{\mathrm{neck}}^{-1/2} s_1, \eps_{\mathrm{neck}}^{1/2} s_2}(q)$ where $\psi_s, \Gamma_s$ may a-priori depend on $s$.
\end{claim}

\begin{proof}
Assume towards a contradiction that the claim is not true. Then for given $n\geq 4$, $r\geq 2$, $\underline{\mu} > -\infty$, $k\in\mathbb{N}$ and $\eps>0$ there exist sequences $\eps_i\to 0$, $\sigma_i\to 0$ and a family of complete $n$-dimensional Ricci shrinkers $(M_i,g_i,f_i)$ containing annuli $A_{s^i_1, s^i_2}(q_i)$ satisfying the assumptions of the theorem but containing some sub-annuli $A_{\frac{1}{2}s_i,s_i}(q_i)$ (with $s_i\in [2\eps_i^{-1/2}s^i_1, \eps_i^{1/2}s^i_2]$) that, after rescaling the metric by $\wt{g}_i:=s_i^{-2}g_i$, are \emph{not} $\eps$-close in the $C^k$ topology to an annular portion of \emph{any} cone $\mathcal{C}(S^{n-1}/\Gamma)$.

By Theorem \ref{blow-up.thm}, we can take a pointed orbifold Cheeger--Gromov limit of $(M_i,\wt{g}_i,f_i,q_i)$ converging to an orbifold $(V,h,q_\infty)$. By the condition $s_i\in [2\eps_i^{-1/2}s^i_1, \eps_i^{1/2}s^i_2]$, we obtain for every $\ell \in \mathbb{N}$ that for sufficiently large $i$ we have $A_{\frac{1}{2\ell}s_i, \ell s_i}(q_i) \subset A_{s^i_1, s^i_2}(q_i)$ and therefore
\begin{equation*}
\int_{A_{\frac{1}{2\ell}s_i, \ell s_i}(q_i)} \abs{\Rm_{g_i}}^{n/2}_{g_i} dV_{g_i} \leq \eps_i \to 0.
\end{equation*}
In particular, after rescaling $\wt{g}_i:=s_i^{-2}g_i$, there are no points of bad convergence on the annulus $A^{\wt{g}_i}_{(2\ell)^{-1}\!,\, \ell}(q_i)$. Repeating this for larger and larger $\ell$, we see that the convergence is smooth away from $q_\infty$, i.e. $(V,h)$ has at most one orbifold point. Moreover, using the argument from \eqref{eq.rescaledepsreg} applied to larger and larger balls, we obtain that the limit is \emph{flat}.

Following the proof of Theorem 1.8 in \cite{AC91_diffeofin}, respectively Section 5 of \cite{BKN89_alecoords} (which we can certainly do because the Ricci-flatness of the limit gives pointwise Ricci bounds on the annuli $A^{\wt{g}_i}_{(2\ell)^{-1}\!,\, \ell}(q_i)$), we see that therefore the limit $(V,h)$ must be an Euclidean cone, i.e. there exists some $\Gamma \subset O(n)$ acting freely on $S^{n-1}$ such that $(V,h) = \mathcal{C}(S^{n-1}/\Gamma)$. As the convergence is smooth away from the origin, we obtain the desired contradiction and hence the claim holds true.
\end{proof}

Having obtained $\psi_s$ and groups $\Gamma_s$, all that remains is to rule out the possible $s$ dependence. This is the second step of the proof.
\begin{claim}
The subgroup $\Gamma_s$ is independent of $s$ and, after slight modifications, some (or all) of the maps $\psi_s$ can be combined to yield the map $\psi$ in the statement of the theorem.
\end{claim}

\begin{proof}
We know there are constants $\eps_{\mathrm{neck}}, \sigma_1, \gamma$ such that for any $s$ with $2\eps_{\mathrm{neck}}^{-1/2} s_1 \leq s \leq \eps_{\mathrm{neck}}^{1/2}s_2$, the annulus $A_{\frac{1}{2}s,s}(q)$ is $\eps$-quasi-isometric and $\eps s^{-k}$-close in the $C^k$ sense to an annular region in a cone, $\mathcal{C}_{\frac{1}{2}s,s}(S^{n-1}/\Gamma_s)$. Now take $\eps$ sufficiently small and, for some fixed $s$, take $s'$ very close to $s$. On its maximal domain of definition $\psi^{-1}_{s'} \circ \psi_s$ is a $2\eps$-quasi-isometry. Therefore $\Gamma_s$ is locally constant and thus independent of $s$.

Once we know that the cone $\mathcal{C}(S^{n-1}/\Gamma)$ is fixed, we can let $t_i := (\frac{2}{3})^{i-1}\eps_{\mathrm{neck}}^{1/2}s_2$ and set $\psi_i = \psi_{t_i}$. These maps $\psi_i$ almost agree after radial scaling and hence, after a further slight modification, can then be piece-wisely connected to yield the map $\psi$, precisely following the argument from \cite{AC91_diffeofin}, page 241.
\end{proof}

Combining the two claims, the theorem is proved.	
\end{proof}

\section{Improved Kato Inequality and Energy Estimate in Necks}\label{sec_Kato}
The main estimate of this section, Theorem \ref{no_neck_energy}, which is to some extent inspired by work of Bando and Bando--Kasue--Nakajima on Einstein manifolds \cite{Ba90_bubbling, BKN89_alecoords}, will allow us to show that energy does not concentrate in a neck region during the bubble tree construction.

The following proposition, on which the energy estimate from Theorem \ref{no_neck_energy} is based, can be seen as a purely analytical result, which requires only the uniform local Sobolev constant bounds from Lemma \ref{Sobolev_bounds}.

\begin{proposition}[Annulus Estimate]\label{grs_bando_lemma2}
Let $n\geq 4$, $r\geq 2$, $\underline{\mu} > -\infty$ and $\alpha>1$ be given constants. Then there exist $\eps_{\mathrm{ann}}>0$, $\sigma_2>0$ and $C_2<\infty$ such that the following holds.

Let $(M,g,f)$ be an $n$-dimensional gradient Ricci shrinker with $\mu(g) \geq \underline{\mu}$. Take $q \in B_g(p,r+1)$ where $p := \argmin_M f$ and let $A_{s_1, s_2}(q) \subset M$ be the unique connected component of the geodesic annulus $\bar{A}_{s_1, s_2}(q)$ which satisfies the condition $A_{s_1, s_2}(q) \cap \partial B_g(q,s_2) \neq \emptyset$ (according to Lemma \ref{small_annuli}) and with
\begin{equation*}
s_2 \leq \sigma_2, \qquad s_1 \leq \tfrac{1}{4} s_2.
\end{equation*}
Finally, let $u,v$ be non-negative functions such that $\Lap_f u = \Lap u - \scal{\nabla f,\nabla u} \geq - uv$ and suppose that $v \in L^{\frac{n}{2}}$ with
\begin{equation}\label{eq.smallv}
\int_{A_{s_1,s_2}(q)} v^{\frac{n}{2}} dV_g \leq \eps_{\mathrm{ann}}
\end{equation}
and $u \in L^\alpha$. Then for $\gamma=\frac{n}{n-2}$, we have 
\begin{equation*}
\int_{A_{s_1,s_2}(q)} u^{\alpha\gamma} dV_g \leq C_2 \int_{A_{s_1, 2s_1}(q)\, \cup\,  A_{\frac{1}{2}s_2, s_2}(q)}u^{\alpha\gamma} dV_g.
\end{equation*}
\end{proposition}

\begin{proof}
The first part of the proof (up to \eqref{eq.Mosteritstep1} below) is related to the first step in a standard Moser iteration or epsilon regularity argument with some extra work to take care of the $\nabla f$ terms coming from the drift Laplacian.

We will work with a cutoff function $0\leq \varphi \leq 1$ with compact support in $A_{s_1,s_2}(q) \subset B_g(p,2r)$ which we will determine more precisely further below. We have
\begin{align*}
-\int_M \varphi^2 u^{\alpha-1}\Lap u \, dV_g &= \int_M \varphi^2\nabla u^{\alpha-1} \nabla u \, dV_g + \int_M \nabla\varphi^2 \cdot u^{\alpha-1} \nabla u \, dV_g\\
&= \tfrac{4(\alpha-1)}{\alpha^2} \int_M \varphi^2\abs{\nabla u^{\alpha/2}}^2 dV_g + \tfrac{4}{\alpha} \int_M \varphi\nabla \varphi \cdot u^{\alpha/2}\nabla u^{\alpha/2} dV_g
\end{align*}
Rearranging this and applying the differential inequality $\Lap_f u = \Lap u - \scal{\nabla f,\nabla u} \geq -uv$, we get
\begin{equation*}
\tfrac{4(\alpha-1)}{\alpha^2} \int_M \varphi^2 \abs{\nabla u^{\alpha/2}}^2 dV_g \leq \int_M \varphi^2 u^\alpha v \, dV_g - \int_M \varphi^2 u^{\alpha-1} \scal{\nabla f, \nabla u}  dV_g - \tfrac{4}{\alpha}\int_M \varphi\nabla \varphi \cdot u^{\alpha/2}\nabla u^{\alpha/2} dV_g.
\end{equation*}
Using Young's inequality, the last term on the right hand side can be estimated by
\begin{equation*}
- \tfrac{4}{\alpha}\int_M \varphi\nabla \varphi \cdot u^{\alpha/2}\nabla u^{\alpha/2} dV_g \leq \tfrac{2}{\alpha}\bigg[\tfrac{\alpha-1}{\alpha} \int_M \varphi^2 \abs{\nabla u^{\alpha/2}}^2 dV_g + \tfrac{\alpha}{\alpha-1} \int_M \abs{\nabla\varphi}^2 u^\alpha dV_g\bigg].
\end{equation*}
Hence, after absorption, we find
\begin{equation}\label{eq.propproof1}
\tfrac{2(\alpha-1)}{\alpha^2} \int_M \varphi^2 \abs{\nabla u^{\alpha/2}}^2 dV_g \leq \int_M \varphi^2 u^\alpha v \, dV_g + \tfrac{2}{\alpha-1}\int_M \abs{\nabla\varphi}^2 u^\alpha dV_g - \int_M \varphi^2 u^{\alpha-1} \scal{\nabla f, \nabla u}  dV_g.
\end{equation}
Let us now estimate the term involving $\nabla f$. Integrating by parts yields
\begin{align*}
- \int_M \varphi^2 u^{\alpha-1} \scal{\nabla f, \nabla u}  dV_g &= \int_M u^\alpha \scal{\nabla \varphi^2 , \nabla f } dV_g + \int_M \varphi^2 u \scal{\nabla u^{\alpha-1}, \nabla f}  dV_g + \int_M \varphi^2 u^\alpha \Lap f  \, dV_g\\
&= \int_M u^\alpha \scal{\nabla \varphi^2 , \nabla f } dV_g + (\alpha-1)\int_M \varphi^2 u^{\alpha-1} \scal{\nabla f, \nabla u}  dV_g + \int_M \varphi^2 u^\alpha \Lap f  \, dV_g
\end{align*}
and thus after subtracting the second term on the right hand side
\begin{align*}
-\alpha \int_M \varphi^2 u^{\alpha-1} \scal{\nabla f, \nabla u}  dV_g &= \int_M u^\alpha \scal{\nabla \varphi^2 , \nabla f } dV_g + \int_M \varphi^2 u^\alpha \Lap f  \, dV_g.
\end{align*}
We can therefore estimate, using Young's inequality,
\begin{align*}
- \int_M \varphi^2 u^{\alpha-1} \scal{\nabla f, \nabla u} dV_g &= \tfrac{1}{\alpha} \int_M u^\alpha \big(\scal{\nabla\varphi^2,\nabla f}+\varphi^2\Lap f\big) dV_g\\
&= \tfrac{1}{\alpha} \int_M u^\alpha \big(2\varphi\scal{\nabla\varphi,\nabla f}+\varphi^2\Lap f\big) dV_g\\
&\leq \tfrac{1}{\alpha} \int_M \abs{\nabla\varphi}^2 u^\alpha dV_g + \tfrac{1}{\alpha} \int_M \varphi^2 u^\alpha \big(\abs{\nabla f}^2+\Lap f\big) dV_g\\
&\leq \tfrac{1}{\alpha} \int_M \abs{\nabla\varphi}^2 u^\alpha dV_g + \tfrac{C(r)}{\alpha} \int_M \varphi^2 u^\alpha dV_g,
\end{align*}
where $C(r)$ is a bound on $\abs{\nabla f}^2+\Lap f$ inside $B_g(p,2r)$. (Such a bound clearly exists: for $\abs{\nabla f}^2$ we have derived it in \eqref{eq.nablafest} and, using the trace of the shrinker equation \eqref{eq.shrinker} and the fact that $R_g \geq 0$, we also have $\Lap f \leq \tfrac{n}{2}$ everywhere.) Plugging this last estimate into \eqref{eq.propproof1}, we obtain
\begin{align*}
\int_M \varphi^2 \abs{\nabla u^{\alpha/2}}^2 dV_g &\leq \tfrac{\alpha^2}{2(\alpha-1)}\bigg[\int_M \varphi^2 u^\alpha v\, dV_g+ \big(\tfrac{2}{\alpha-1}+\tfrac{1}{\alpha}\big)\int_M \abs{\nabla \varphi}^2 u^\alpha dV_g+ \tfrac{C(r)}{\alpha}\int_M \varphi^2 u^\alpha dV_g\bigg]\\
&\leq C \int_M \varphi^2 u^\alpha v + \abs{\nabla \varphi}^2 u^\alpha + \varphi^2 u^\alpha dV_g,
\end{align*}
where 
\begin{equation}\label{eq.valueofC}
C=C(n,r,\underline{\mu},\alpha)= \tfrac{\alpha^2}{2(\alpha-1)}\max\Big\{1,\tfrac{2}{\alpha-1}+\tfrac{1}{\alpha}, \tfrac{C(r)}{\alpha}\Big\}. 
\end{equation}
Next, combining this estimate with the uniform Sobolev inequality from Lemma \ref{Sobolev_bounds} (which we can apply if $\sigma_2 \leq \delta_0(2r)$) and the smallness assumption \eqref{eq.smallv}, and noting that $2^*=2\gamma$, we conclude
\begin{align*}
\left(\int_M \big(\varphi u^{\alpha/2}\big)^{2\gamma}dV_g\right)^{\frac{1}{\gamma}} &\leq C_S \int_M \abs{\nabla\left(\varphi u^{\alpha/2}\right)}^2 dV_g\\
&\leq C_S \int_M \abs{\nabla \varphi}^2 u^\alpha + \varphi^2\abs{\nabla u^{\alpha/2}}^2 dV_g\\
&\leq C_S (C+1) \int_M \varphi^2 u^\alpha v + \abs{\nabla \varphi}^2 u^\alpha + \varphi^2 u^\alpha dV_g\\
&\leq C_S (C+1) \bigg[\left(\int_M v^{\frac{n}{2}}dV_g\right)^{\frac{2}{n}}\left(\int_M \varphi^{2\gamma} u^{\alpha\gamma}dV_g\right)^{\frac{1}{\gamma}} + \int_M \varphi^2 u^\alpha + \abs{\nabla \varphi}^2 u^\alpha dV_g\bigg]\\
&\leq C_S (C+1) \bigg[\eps_{\mathrm{ann}}^{2/n}\left(\int_M \big(\varphi u^{\alpha/2}\big)^{2\gamma}dV_g\right)^{\frac{1}{\gamma}} + \int_M \varphi^2 u^\alpha + \abs{\nabla \varphi}^2 u^\alpha dV_g\bigg].
\end{align*}
We can absorb the first term on the last line by taking $\eps_{\mathrm{ann}}$ small enough. For example, letting $\eps_{\mathrm{ann}}^{2/n} \leq \frac{1}{2C_S(C+1)}$, we obtain
\begin{equation}\label{eq.Mosteritstep1}
\left(\int_M \big(\varphi u^{\alpha/2}\big)^{2\gamma}dV_g\right)^{\frac{1}{\gamma}} \leq 2C_S(C+1) \int_M \varphi^2 u^\alpha + \abs{\nabla \varphi}^2 u^\alpha dV_g.
\end{equation}

Now, choose $0\leq \varphi \leq 1$ so that $\varphi = 1$ on $A_{2s_1,\frac{1}{2}s_2}(q)$, $\varphi = 0$ on $M \setminus A_{s_1,s_2}\left(q\right)$, and 
\begin{equation}\label{eq.nablaphiassumption}
\abs{\nabla \varphi} \leq 
\begin{cases}
\frac{C'}{s_1} & \text{on } A_{s_1,2s_1}(q) \\
\frac{C'}{s_2} & \text{on } A_{\frac{1}{2}s_2, s_2}(q)
\end{cases}
\end{equation}
for some universal constant $C'<\infty$. Using \eqref{eq.Mosteritstep1}, we get
\begin{align*}
\bigg(\int_{A_{s_1,s_2}(q)} u^{\alpha\gamma} dV_g\bigg)^{\frac{1}{\gamma}} &\leq \bigg(\int_{A_{s_1,2s_1}(q)\, \cup \, A_{\frac{1}{2}s_2,s_2}(q)} u^{\alpha\gamma} dV_g\bigg)^{\frac{1}{\gamma}} + \bigg(\int_{A_{2s_1,\frac{1}{2}s_2}(q)} \big(u^{\alpha/2}\big)^{2\gamma}dV_g\bigg)^{\frac{1}{\gamma}}\\
&\leq \bigg(\int_{A_{s_1,2s_1}(q)\, \cup \, A_{\frac{1}{2}s_2,s_2}(q)} u^{\alpha\gamma} dV_g\bigg)^{\frac{1}{\gamma}} + \bigg(\int_M \big(\varphi u^{\alpha/2}\big)^{2\gamma}dV_g\bigg)^{\frac{1}{\gamma}}\\
&\leq \bigg(\int_{A_{s_1,2s_1}(q)\, \cup \, A_{\frac{1}{2}s_2,s_2}(q)} u^{\alpha\gamma} dV_g\bigg)^{\frac{1}{\gamma}} + 2C_S(C+1) \int_M \varphi^2 u^\alpha + \abs{\nabla \varphi}^2 u^\alpha dV_g.
\end{align*}
H\"older's inequality yields
\begin{equation*}
\int_M \varphi^2 u^\alpha dV_g \leq \Vol^{\frac{2}{n}}(A_{s_1,s_2}(q)) \cdot \bigg(\int_{A_{s_1,s_2}(q)} u^{\alpha\gamma} dV_g\bigg)^{\frac{1}{\gamma}} 
\end{equation*}
Hence, if $\sigma_2$ is chosen sufficiently small such that $\Vol^{\frac{2}{n}}(A_{s_1,s_2}(q)) \leq \frac{1}{4C_S(C+1)}$ -- which can be done due to the uniform volume growth estimate \eqref{eq.volumegrowth} -- then this term can be absorbed, leading to
\begin{equation*}
\bigg(\int_{A_{s_1,s_2}(q)} u^{\alpha\gamma} dV_g\bigg)^{\frac{1}{\gamma}} \leq 2\bigg(\int_{A_{s_1,2s_1}(q)\, \cup \, A_{\frac{1}{2}s_2,s_2}(q)} u^{\alpha\gamma} dV_g\bigg)^{\frac{1}{\gamma}} + 4C_S(C+1) \int_M \abs{\nabla \varphi}^2 u^\alpha dV_g.
\end{equation*}
Finally, applying H\"older's inequality also to the last term, we find for some $C''<\infty$
\begin{equation*}
\int_M \abs{\nabla \varphi}^2 u^\alpha dV_g \leq \left(\int_{\mathrm{supp}(\nabla\varphi)} \abs{\nabla \varphi}^n dV_g\right)^{\frac{2}{n}}\left(\int_{\mathrm{supp}(\nabla\varphi)} u^{\alpha\gamma} dV_g\right)^{\frac{1}{\gamma}} \leq C'' \bigg(\int_{A_{s_1,2s_1}(q)\, \cup \, A_{\frac{1}{2}s_2,s_2}(q)} u^{\alpha\gamma} dV_g\bigg)^{\frac{1}{\gamma}}.
\end{equation*}
Here, we have used the volume growth estimate \eqref{eq.volumegrowth} and the assumption \eqref{eq.nablaphiassumption} for the last estimate. We therefore conclude
\begin{equation*}
\bigg(\int_{A_{s_1,s_2}(q)} u^{\alpha\gamma} dV_g\bigg)^{\frac{1}{\gamma}} \leq (2+4C_S(C+1)C'')\bigg(\int_{A_{s_1,2s_1}(q)\, \cup \, A_{\frac{1}{2}s_2,s_2}(q)} u^{\alpha\gamma} dV_g\bigg)^{\frac{1}{\gamma}}
\end{equation*}
and hence the proposition is proved with $C_2 = (2+4C_S(C+1)C'')^\gamma$.
\end{proof}

Endowed with this proposition, we would now like to show that for small annuli $A_{s_1, s_2}(q)$ (under assumptions similar to the ones in the neck theorem), the energy of the entire annulus can be estimated by the energy of the two dyadic annuli $A_{s_1, 2s_1}(q)$ and $A_{\frac{1}{2}s_2, s_2}(q)$. It is tempting to use the equation
\begin{equation}\label{eq.wrongequation}
\Lap_f \abs{\Rm} \geq - C\abs{\Rm}^2,
\end{equation}
and try to apply Proposition \ref{grs_bando_lemma2} to $u=\abs{\Rm}$, $v=C\abs{\Rm}$ with $\alpha\gamma=\frac{n}{2}$, but unfortunately, this does \emph{not} work: For example if $n=4$, we have $\gamma=\frac{n}{n-2}=2=\frac{n}{2}$, so would need to work with $\alpha=1$, but the proposition crucially needs $\alpha>1$ and, as can be clearly seen from \eqref{eq.valueofC}, the constant $C_2$ degenerates as $\alpha \searrow 1$. It is therefore necessary to improve the differential inequality \eqref{eq.wrongequation}, which we will do in the following. A key ingredient for this is the following improved Kato inequality for gradient Ricci shrinkers.
\begin{lemma}[Improved Kato Inequality]\label{Br00_kato_grs}
There exists a constant $\delta_K = \delta_K(n)>0$ such that the following holds. If $(M,g,f)$ is an $n$-dimensional oriented gradient Ricci shrinker, then
\begin{equation*}
(1+\delta_K)\abs{\nabla \abs{\Rm}}^2 \leq \abs{\nabla \Rm}^2.
\end{equation*}
\end{lemma}

\begin{proof}
One can deduce an improved Kato inequality from an explicit calculation similar to the work of Bando--Kasue--Nakajima \cite{BKN89_alecoords} in the Einstein case. Here however, we will rely on a general framework, due to Branson \cite{Br00_kato} (see also Calderbank--Gauduchon--Herzlich \cite{CGH00_kato} for a similar result with a quite different proof), for determining when an improved Kato inequality holds on an oriented manifold. Specifically, in order to apply Theorem $4$ in \cite{Br00_kato} we need to consider a first order operator $D$ and a tensor bundle $T$ with sections $\psi$. Then, if $D^\ast D$ is elliptic when acting on $T$ and $D\psi = 0$ we will have an improved Kato inequality for $\psi$ away from its zero set. Such conditions are typically satisfied for a curvature tensor because of the Bianchi identites and some extra structure, which in our case is the shrinker equation.

Branson's framework requires one to work with an operator $D$ which is the sum of generalised gradients (also called Stein--Weiss operators), two examples of which are the exterior derivative $d$ and its adjoint $d^\ast$ acting on differential forms, see \cite{SW68}. Viewing $\Rm$ as a vector bundle valued $2$-form in $\Omega^2\left(M, \mathrm{End}\left(TM\right)\right)$ and taking $d$ to be the exterior \textit{covariant} derivative, we note that $d \Rm = 0$ by the second Bianchi identity. 

Instead of $d^\ast$, we would like to work with $d^\ast_f = -\div_f = - e^f \div (e^{-f} \cdot)$, the adjoint of $d$ with respect to $e^{-f}dV_g$. This is the natural adjoint to work with in the shrinker setting, but it is not immediately clear if it is a Stein--Weiss operator. However we can use that forms are linear with respect to smooth functions which gives
\begin{equation*}
\div_f \Rm\left(\cdot\right) = e^f \div\left(e^{-f}\Rm\left(\cdot,\cdot\right)\right) = e^f \div\left(\Rm\left(e^{-f}\cdot, \cdot\right)\right).
\end{equation*} 
Thus we are actually dealing with $\div$, or equivalently $d^\ast$, which we know is Stein--Weiss, except now we have applied a transformation to the domain of $\Rm$. However, since this transformation is confromal the new domain is isomorphic to $TM$. This is enough for our purposes, since being Stein--Weiss is an algebraic property. Using the second Bianchi identity, the shrinker quation \eqref{eq.shrinker}, and the commutator rule we can compute the following in coordinates:
\begin{equation}\label{eq.weighteddiv}
\begin{aligned}
\div \Rm &= \nabla_p \Rm_{ij\ell p}\\
&= \nabla_j \Ric_{i\ell} - \nabla_i \Ric_{j\ell}\\
&= \nabla_j \big(\tfrac{1}{2}g_{i\ell} - \nabla_i \nabla_\ell f\big) - \nabla_i \big(\tfrac{1}{2}g_{j\ell} - \nabla_j \nabla_\ell f\big)\\
&= -\nabla_j \nabla_i \nabla_\ell f + \nabla_i \nabla_j \nabla_\ell f \\
&= \Rm_{ij\ell p}\nabla_p f.
\end{aligned}
\end{equation}
This is equivalent to $d^\ast_f \Rm = 0$.

With all of this in mind, we take $D = d + d^\ast_f$ and have $D \Rm = 0$. This also gives $D^\ast D = \left(d + d^\ast_f\right)^2 = \Lap^H_f$ where $\Lap^H_f$ is the f-Hodge Laplacian, which is certainly elliptic. Therefore we can apply Theorem $4$ in \cite{Br00_kato} to get the desired improved Kato inequality away from the set of points where $\Rm = 0$. However, such a set is empty on a non-trivial shrinker. This completes the proof.

To be precise, in dimension $n=4$, we need to split $2$-forms into their self-dual and anti-self-dual parts in order to obtain Stein--Weiss operators $d_{\pm}$ and $(d^\ast_f)_{\pm}$, see Branson's work in \cite{Br97_SW} for details.
\end{proof}

As a corollary, we obtain the following improvement of \eqref{eq.wrongequation}.
\begin{corollary}[Improved Differential Inequality for the Riemann Tensor on a Ricci Shrinker]\label{grs_rm_ineqs}
There exists a constant $C_K=C_K(n)<\infty$ such that for every $n$-dimensional oriented gradient Ricci shrinker $(M,g,f)$ and $\delta_K(n)$ from the improved Kato inequality, we have
\begin{equation}\label{eq.improveddiffineq}
\Lap_f \abs{\Rm}^{1-\delta_K} \geq -C_K \abs{\Rm}^{2-\delta_K},
\end{equation}
where $\Lap_f u = \Lap u - \scal{\nabla f,\nabla u}$ denotes the drift Laplacian on $(M,g,f)$.
\end{corollary}

\begin{proof}
The proof is in two steps. We first show the following shrinker version of the evolution equation of the Riemann tensor along the Ricci flow.
\begin{claim}\label{rm_drift_evo}
The Riemann tensor on a gradient Ricci shrinker $(M,g,f)$ satisfies the following equation
\begin{equation*}
\Lap_f \Rm = \Rm + Q(\Rm),
\end{equation*}
where $Q(\Rm)$ is a quadratic expression in $\Rm$.
\end{claim}

\begin{proof}
In the argument below, the quadratic expression $Q(\Rm)$ may change from line to line. Working in coordinates, we first note that using the commutator rule, the second Bianchi identity, and \eqref{eq.weighteddiv} we have
\begin{equation}\label{eq.rmevo1}
\begin{aligned}
\nabla_p \nabla_p \Rm_{ijk \ell} &= -\nabla_p \nabla_k \Rm_{ij \ell p} - \nabla_p \nabla_\ell \Rm_{ijpk}\\
&= -\nabla_k \nabla_p \Rm_{ij \ell p} - \nabla_\ell \nabla_p \Rm_{ijpk} + Q(\Rm)\\
&= \nabla_k \Rm_{ji\ell p}\nabla_p f + \nabla_\ell \Rm_{ijkp}\nabla_p f + \Rm_{ji\ell p}\nabla_k \nabla_p f + \Rm_{ijkp} \nabla_\ell \nabla_p f + Q(\Rm).
\end{aligned}
\end{equation}
Using the second Bianchi identity for the terms involving first derivatives of the Riemann tensor yields $(\nabla_k \Rm_{ji \ell p} + \nabla_\ell \Rm_{ijkp})\nabla_p f = \nabla_p \Rm_{ijk\ell} \nabla_p f$. The terms involving second derivatives of the shrinker potential are handled using the shrinker equation \eqref{eq.shrinker} one last time:
\begin{align*}
\Rm_{ji \ell p}\nabla_k \nabla_p f + \Rm_{ijkp} \nabla_\ell \nabla_p f &= \Rm_{ji \ell p}\big(\tfrac{1}{2}g_{kp} - \Ric_{kp}\big) + \Rm_{ijkp}\big(\tfrac{1}{2}g_{\ell p} - \Ric_{\ell p}\big)\\
&= \Rm_{ijk \ell} + Q(\Rm).
\end{align*}
Putting everything together, the claim follows.
\end{proof}

Using the identity $\nabla \abs{\Rm} = \abs{\Rm}^{-1}\scal{\nabla \Rm, \Rm}$ and the improved Kato inequality
\begin{align*}
\Lap_f \abs{\Rm}^{1-\delta_K} &= (1-\delta_K) \nabla \cdot (\scal{\nabla \Rm, \Rm} \abs{\Rm}^{-1-\delta_K}) - (1-\delta_K)\scal{\scal{\nabla f, \nabla \Rm}, \Rm }\abs{\Rm}^{-1-\delta_K}\\
&= (1-\delta_K) \scal{\Lap \Rm, \Rm} \abs{\Rm}^{-1-\delta_K} - (1-\delta_K)\scal{\scal{\nabla f, \nabla \Rm}, \Rm }\abs{\Rm}^{-1-\delta_K}\\
&\quad+ (1-\delta_K)\abs{\nabla \Rm}^2 \abs{\Rm}^{-1-\delta_K} - (1-\delta_K)(1+\delta_K)\scal{\nabla \Rm, \Rm}\nabla \abs{\Rm} \abs{\Rm}^{-2-\delta_K}\\
&\geq (1-\delta_K) \scal{\Lap_f \Rm, \Rm} \abs{\Rm}^{-1-\delta_K}.
\end{align*}
Thus, using Claim \ref{rm_drift_evo} as well as the fact that $Q(\Rm)\geq - \bar{C}\abs{\Rm}^2$ for some constant $\bar{C}$, we find
\begin{align*}
\Lap_f \abs{\Rm}^{1-\delta_K} &\geq (1-\delta_K)\scal{\Rm + Q(\Rm), \Rm}\abs{\Rm}^{-1-\delta_K}\\
&= (1-\delta_K)\abs{\Rm}^{1-\delta_K} + (1-\delta_K)\scal{Q(\Rm), \Rm}\abs{\Rm}^{-1-\delta_K}\\
&\geq -(1-\delta_K)\bar{C} \abs{\Rm}^{2-\delta_K}
\end{align*}
Hence the corollary follows by setting $C_K := (1-\delta_K)\bar{C}$.
\end{proof}

We can now combine this improved differential inequality with Proposition \ref{grs_bando_lemma2} to obtain an energy estimate in neck regions for oriented gradient shrinkers as desired. This is the main result of this section.
\begin{theorem}[Energy Estimate in Necks for Ricci Shrinkers]\label{no_neck_energy}
Given $n \geq 4$, $r\geq 2$, and $\underline{\mu} > -\infty$, there exist $\eps_{\mathrm{ee}}>0$, $\sigma_3>0$ and $C_3<\infty$ such that the following holds.

Let $(M,g,f)$ be an $n$-dimensional oriented gradient Ricci shrinker with $\mu(g) \geq \underline{\mu}$. Take $q \in B_g(p,r+1)$ where $p := \argmin_M f$ and let $A_{s_1, s_2}(q) \subset M$ be the unique connected component of the geodesic annulus $\bar{A}_{s_1, s_2}(q)$ which satisfies the condition $A_{s_1, s_2}(q) \cap \partial B_g(q,s_2) \neq \emptyset$ (according to Lemma \ref{small_annuli}) and with
\begin{equation}\label{eq.annuluscondition1}
s_2 \leq \sigma_3, \qquad s_1 \leq \tfrac{1}{4} s_2.
\end{equation}
Finally, assume that
\begin{equation}\label{eq.annuluscondition2}
\int_{A_{s_1, s_2}(q)} \abs{\Rm_g}^{n/2}_g dV_g \leq \eps_{\mathrm{ee}}.
\end{equation}
Then we have 
\begin{equation*}
\int_{A_{s_1,s_2}(q)} \abs{\Rm_g}^{n/2}_g dV_g \leq C_3 \int_{A_{s_1, 2s_1}(q)\, \cup\,  A_{\frac{1}{2}s_2, s_2}(q)}\abs{\Rm_g}^{n/2}_g dV_g.
\end{equation*}
\end{theorem}

\begin{proof}
We set $u:=\abs{\Rm}^{1-\delta_K}$ and $v:=C_K\abs{\Rm}$, where $\delta_K$ and $C_K$ are from Corollary \ref{grs_rm_ineqs}. Then \eqref{eq.improveddiffineq} is equivalent to $\Lap_f u \geq -uv$. Letting $\sigma_3=\sigma_2$ and $\eps_{\mathrm{ee}} = C_K^{-n/2}\eps_{\mathrm{ann}}$ (with $\sigma_2$ and $\eps_{\mathrm{ann}}$ given by Proposition \ref{grs_bando_lemma2}), we have
\begin{equation*}
\int_{A_{s_1,s_2}(q)} v^{\frac{n}{2}} dV_g \leq C_K^{n/2}\eps_{\mathrm{ee}} \leq \eps_{\mathrm{ann}}
\end{equation*}
and moreover for $\alpha:=\frac{n-2}{2(1-\delta_K)}>1$ the fact that $\abs{\Rm}\in L^{\frac{n}{2}}$ (and therefore by H\"older's inequality $\abs{\Rm}\in L^{\frac{n-2}{2}}$) shows that $u\in L^\alpha$. We can therefore apply Proposition \ref{grs_bando_lemma2} which yields the claimed estimate as $u^{\alpha\gamma}=\abs{\Rm}^{n/2}$ with $C_3=C_2(\alpha=\frac{n-2}{2(1-\delta_K)})$.
\end{proof}


\section{Construction of the Bubble Tree}\label{sec_bubbling}
It is finally time to construct the bubble tree and prove Theorem \ref{shrinker_bubbling}. So let $n\geq 4$ and let $(M_i,g_i,f_i)$ be a sequence of $n$-dimensional oriented gradient Ricci shrinkers with entropy uniformly bounded below $\mu(g_i)\geq\underline{\mu}>-\infty$ and basepoints $p_i=\argmin_M f_i$. If $n>4$, then additionally assume \eqref{eq.energybounds} -- recall that for $n=4$ this is always satisfied automatically. Finally, we also fix a small $\eps>0$, $k\in\mathbb{N}$, and $r\geq 2$ such that $\sQ \cap \partial B_{g_\infty}(p_\infty,r) = \emptyset$ and let $\eps_{\mathrm{neck}}$, $\sigma_1$ and $\gamma$ be the corresponding constants from Theorem \ref{shrinker_neck_thm}.

By the arguments from Section \ref{sec_blowup} and in particular \eqref{eq.energyconcentration}, we know that for each $q^\ell\in\sQ_r$ there are $M_i \ni q^\ell_i \to q^\ell$ such that the convergence of $B_{g_i}(p_i,r) \setminus \bigcup_{\ell} B_{g_i}(q^\ell_i,\delta)$ is smooth for any sufficiently small $\delta<<\delta_0$. In particular, we obtain
\begin{equation}\label{eq.energywithoutdeltaballs}
\lim_{\delta\to 0}\, \lim_{i \to \infty} \int_{B_{g_i}(p_i,r) \setminus \bigcup_{\ell} B_{g_i}(q^\ell_i,\delta)} \abs{\Rm_{g_i}}^{n/2}_{g_i} dV_{g_i} = \int_{B_{g_\infty}(p_\infty,r)} \abs{\Rm_{g_\infty}}^{n/2}_{g_\infty} dV_{g_\infty}.
\end{equation}
In the following, we investigate what happens inside the $\delta$-balls. In the following argument, we may only consider $\delta$ sufficiently small so that the regions $B_{g_\infty}(q^\ell,10\cdot \delta) \setminus \{q^\ell\}$ do not contain any other orbifold points, and $i$ sufficiently large so that all $B_{g_i}(q^\ell_i,\delta)$ are disjoint. This allows us to focus on a single orbifold point $q$. 

Given such a point $q \in \sQ_r$, we fix a corresponding sequence $M_i \ni q_i\to q$ along which the curvature concentrates in the sense of \eqref{eq.energyconcentration}. The task is then to extract a (finite) number of point-scale sequences that detect all the ALE bubbles that form at $q$.

\paragraph{The first bubble:}
Let $\bar{\eps}:=\min\{\eps_{\mathrm{reg}},\eps_{\mathrm{gap}},\eps_{\mathrm{neck}}, \eps_{\mathrm{ee}}\}$ where $\eps_{\mathrm{reg}}$ is the constant from the $\eps$-regularity result (Lemma \ref{shrinker_eps_reg}), $\eps_{\mathrm{gap}}$ is from Bando's gap result (Proposition \ref{prop.Bandogap}), $\eps_{\mathrm{neck}}$ has been chosen above as in the neck theorem (Theorem \ref{shrinker_neck_thm}), and $\eps_{\mathrm{ee}}$ is from the energy estimate in necks (Theorem \ref{no_neck_energy}). Set
\begin{equation*}
r_i^1 := \inf \Big\{ r>0 \, \Big\mid \int_{B_{g_i}(q,r)} \abs{\Rm_{g_i}}^{n/2}_{g_i} dV_{g_i} \geq \frac{\bar{\eps} }{2} \text{ for some } B_{g_i}(q,r) \subseteq B_{g_i}(q_i,\delta) \Big\}
\end{equation*}
and let $q_i^1$ be points in $M_i$ such that $B_{g_i}(q_i^1,r_i^1) \subseteq B_{g_i}(q_i,\delta)$ and
\begin{equation*}
\int_{B_{g_i}(q_i^1,r_i^1)} \abs{\Rm_{g_i}}^{n/2}_{g_i} dV_{g_i} \geq \frac{\bar{\eps} }{2}.
\end{equation*}
	
Clearly $r^1_i \to 0$, otherwise there is no curvature concentration as described by \eqref{eq.energyconcentration}. By Theorem \ref{blow-up.thm} the rescaled sequence $(M_i, \wt{g}_i = (r_i^1)^{-2}g_i, q_i^1)$ subconverges in the pointed orbifold Cheeger--Gromov sense to a complete, non-compact, Ricci-flat limit $(V^1, h^1, q_\infty^1)$ with bounded $L^{n/2}$ Riemannian curvature and which is ALE of order $n-1$ in general and ALE of order $n$ if either $n = 4$ or $(V^1,h^1)$ is K\"ahler. By Corollary \ref{cor.oneend}, $(V^1,h^1)$ has one end. Moreover, by the choice of $r_i^1$, any ball of radius $r\leq1$ with respect to the rescaled metric $(r_i^1)^{-2}g_i$ (and contained in $B_{\tilde{g}_i}(q_i,(r^1_i)^{-1}\delta)$) has energy at most $\bar{\eps} /2$ and hence the convergence and the limit are smooth everywhere by the characterisation of singular points (respectively points of bad convergence) in Theorem \ref{blow-up.thm}. We then conclude that
\begin{equation*}
\int_{B_{h^1}(q_\infty^1,1)} \abs{\Rm_{h^1}}^{n/2}_{h^1} dV_{h^1} \geq \frac{\bar{\eps} }{2}
\end{equation*}
which implies the limit is non-flat and hence a (smooth) ALE bubble as in Definition \ref{ale}. By smooth convergence, we conclude that
\begin{equation}\label{eq.firstbubbleenergy}
\lim_{R\to\infty} \lim_{i\to\infty} \int_{B_{g_i}(q_i^1,Rr_i^1)} \abs{\Rm_{g_i}}^{n/2}_{g_i} dV_{g_i} = \int_{V^1} \abs{\Rm_{h^1}}^{n/2}_{h^1} dV_{h^1}.
\end{equation}

We have now extracted the \emph{deepest} bubble corresponding to the \emph{smallest} scale, motivating the following definition.
\begin{definition}[Leaf and Intermediate Bubbles]\label{bubble_types}
An ALE bubble as in Definition \ref{ale} is called a \emph{leaf bubble} if it is smooth. If instead it has finitely many orbifold singularities it is called an \emph{intermediate bubble}.
\end{definition}

If there is further curvature concentration, we continue to extract more point-scale sequences. We first set
\begin{equation}\label{eq.defN}
N:= \frac{4E(2r)}{\bar{\eps}}
\end{equation}
and note that since $B_{g_i}(q_i,\delta)\subseteq B_{g_i}(p_i,2r)$ contains at most $E(2r)$ energy and our method detects \emph{disjoint} regions containing at least $\bar{\eps}/4$ energy, the process will terminate after a finite number of steps $N_q\leq N$.

\paragraph{The second bubble:}
First, in order to make sure we do not simply find the same bubble again, we pick $K^1 >> 1$ large enough, so that
\begin{equation*}
\int_{V^1 \setminus B_{h^1}(q_\infty^1,K^1)} \abs{\Rm_{h^1}}^{n/2}_{h^1} dV_{h^1} \leq \frac{\bar{\eps} }{10N},
\end{equation*}
with $N$ given by \eqref{eq.defN}, which is possible as $(V^1,h^1)$ has bounded $L^{n/2}$ curvature. From this, we conclude that for any constant $R>K^1$ we have
\begin{equation}\label{eq.notenoughenergy}
\int_{A_{K^1r_i^1, Rr_i^1}(q_i^1)} \abs{\Rm_{g_i}}^{n/2}_{g_i} dV_{g_i} \leq \frac{\bar{\eps} }{8N}
\end{equation}
for sufficiently large $i$. We then set
\begin{equation*}
r_i^2 := \inf \Big\{ r>0 \, \Big\mid \int_{B_{g_i}(q,r) \setminus B_{g_i}(q_i^1,K^1 r_i^1)} \abs{\Rm_{g_i}}^{n/2}_{g_i} dV_{g_i} \geq \frac{\bar{\eps} }{2} \text{ for some } B_{g_i}(q,r) \subseteq B_{g_i}(q_i,\delta) \Big\}
\end{equation*}
and let $q_i^2$ be points in $M_i$ such that $B_{g_i}(q_i^2,r_i^2) \subseteq B_{g_i}(q_i,\delta)$ and
\begin{equation}\label{eq.secondcon}
\int_{B_{g_i}(q_i^2,r_i^2) \setminus B_{g_i}(q_i^1,K^1 r_i^1)} \abs{\Rm_{g_i}}^{n/2}_{g_i} dV_{g_i} \geq \frac{\bar{\eps} }{2}.
\end{equation}
Note that $r^2_i \geq r^1_i$ by construction. We can assume that $r^2_i \to 0$, otherwise there is no more curvature concentration and the process of extracting point-scale sequences stops. We first claim the following.
\begin{claim}
The point-scale sequences satisfy
\begin{equation}\label{eq.claim}
\frac{r^2_i}{r^1_i} + \frac{d(q^1_i, q^2_i)}{r^2_i} \to \infty.
\end{equation}
\end{claim}
\begin{proof}
If \eqref{eq.claim} is not true, then there is some number $M > K^1$ such that
\begin{equation*}
1 \leq \frac{r^2_i}{r^1_i} \leq M, \qquad \frac{d(q^1_i, q^2_i)}{r^2_i} \leq M
\end{equation*}
and therefore
\begin{equation*}
\frac{d(q^1_i, q^2_i)}{r^1_i} \leq M^2.
\end{equation*}
This implies that $q^2_i \in B_{g_i}(q^1_i, M^2 r_i^1)$ and thus
\begin{equation*}
B_{g_i}(q_i^2,r_i^2) \setminus B_{g_i}(q_i^1,K^1 r_i^1) \subseteq A_{K^1r_i^1, (M^2+M)r_i^1}(q_i^1).
\end{equation*}
In particular, \eqref{eq.notenoughenergy} and \eqref{eq.secondcon} now yield a contradiction (for $R=M^2+M$) and hence the claim must hold.
\end{proof}

\begin{remark}
An alternative approach is to not pass from $B_{g_i}(q_i^1,r_i^1)$ to the larger balls $B_{g_i}(q_i^1,K^1r_i^1)$ but instead mark and later discard the point-scale sequences that do not satisfy \eqref{eq.claim}. Such a strategy was used in the bubble tree construction in \cite{BS18_minsurf}.
\end{remark}

We then distinguish two cases.

\paragraph{Case 1:}
We have
\begin{equation*}
\frac{d(q^1_i, q^2_i)}{r^2_i} \to \infty.
\end{equation*}

This is the easy case because the bubbles are forming separately. Indeed, the reader can easily verify that if we blow up using $(q^2_i,r^2_i)$ in a similar way as for $(q^1_i,r^1_i)$ above, we get the same conclusion and the first bubble will disappear off at infinity. In particular, we obtain another \emph{leaf bubble}. Clearly, since $r^2_i \geq r^1_i$, we also have 
\begin{equation*}
\frac{d(q^1_i, q^2_i)}{r^1_i} \to \infty,
\end{equation*}
motivating the following definition.
\begin{definition}[Separable Bubbles]
If $(q_i^k,r_i^k)$ and $(q_i^\ell,r_i^\ell)$ are two point-scale sequences such that
\begin{equation*}
\frac{d(q^k_i, q^\ell_i)}{r^k_i} \to \infty \quad\text{and}\quad \frac{d(q^k_i, q^\ell_i)}{r^\ell_i} \to \infty,
\end{equation*}
then we say that the two associated bubbles $(V^k,h^k)$ and $(V^\ell,h^\ell)$ are \emph{separable}.
\end{definition}

\paragraph{Case 2:}
For some $M$, we have
\begin{equation}\label{eq.case2}
\frac{d(q^1_i, q^2_i)}{r^2_i} \leq M < \infty.
\end{equation}

This is the much more delicate case as the bubbles will form on top of each other. We consider the rescaled sequence $(M_i, \wt{g}_i = (r_i^2)^{-2}g_i, q_i^2)$ which by Theorem \ref{blow-up.thm} and Corollary \ref{cor.oneend} subconverges in the pointed orbifold Cheeger--Gromov sense to a complete, non-compact, Ricci-flat limit $(V^2, h^2, q_\infty^2)$ with bounded $L^{n/2}$ Riemannian curvature and with one end satisfying the required ALE condition. The assumption \eqref{eq.case2} shows that, by possibly passing to a further subsequence, $q_i^1$ converge to some $\hat{q}^1_\infty\in V^2$ (with $d(\hat{q}^1_\infty,q_\infty^2)\leq M$). Since by \eqref{eq.claim} and \eqref{eq.case2} we know that $r^1_i/r^2_i \to 0$, we have energy concentration for $\wt{g}_i$ around $q_i^1$ and hence the limit point $\hat{q}^1_\infty \in V^2$ is an orbifold point. By the choice of $r^2_i$, we see that there are no other energy concentrations and hence no further orbifold singularities. 

Note that for any $R\geq K^1$, by the choice of $r^2_i$, we have
\begin{equation*}
\int_{B_{g_i}(q_i^1,\frac{1}{R} r_i^2) \setminus B_{g_i}(q_i^1,R r_i^1)} \abs{\Rm_{g_i}}^{n/2}_{g_i} dV_{g_i} < \frac{\bar{\eps} }{2}.
\end{equation*}
and hence, for sufficiently large $i$, $A_{R r_i^1, \frac{1}{R} r_i^2}(q_i^1)$ satisfies the assumptions \eqref{eq.neckcondition1}--\eqref{eq.neckcondition2} of the neck theorem -- we therefore call such an annulus a \emph{neck region}. We claim the following.

\begin{claim}\label{claim.neck}
No energy is concentrating in the neck region in the following sense:
\begin{equation*}
\lim_{R\to\infty} \lim_{i\to\infty} \int_{A_{R r_i^1, \frac{1}{R} r_i^2}(q_i^1)} \abs{\Rm_{g_i}}^{n/2}_{g_i} dV_{g_i} = 0.
\end{equation*}
\end{claim}

\begin{proof}
It is clear that for $R\to\infty$ and $i\to\infty$ the innermost dyadic sub-annulus $A_{Rr_i^1, 2Rr_i^1}(q_i^1)$ converges smoothly, after rescaling by $2Rr_i^1$, to an annular portion of a flat cone $\mathcal{C}(S^{n-1}/\Gamma_1)$, where $\Gamma_1$ is given by the asymptotic structure of the end of the ALE bubble $(V^1,h^1)$. Similarly, the outermost dyadic sub-annulus $A_{\frac{1}{2R} r_i^2, \frac{1}{R} r_i^2}(q_i)$ converges smoothly, after rescaling by $\frac{1}{R} r_i^2$, to an annular portion of a flat cone $\mathcal{C}(S^{n-1}/\Gamma_2)$ where $\Gamma_2$ is given by the orbifold singularity structure at $\hat{q}^1_\infty$ in $(V^2,h^2)$. In particular, the energy of these inner- and outermost dyadic annuli is converging to zero.

If $R>\eps_{\mathrm{neck}}^{-1/2}$, then by the neck theorem, for large $i$ there exists an $\eps$-quasi isometry from the entire neck region $A_{Rr_i^1, \frac{1}{R}r_i^2}(q_i^1)$ to $\mathcal{C}_{Rr_i^1, \frac{1}{R}r_i^2}(S^{n-1}/\Gamma)$ for some $\Gamma$ with $\abs{\Gamma}<\gamma$. This shows that $\Gamma_1=\Gamma_2=\Gamma$. One might then show that it is possible to let $\eps\to 0$ as $R\to\infty$, so that, after rescaling, one obtains (smooth) convergence of \emph{any} dyadic sub-annulus $A_{\frac{1}{2}s_i,s_i}(q_i^1) \subseteq A_{Rr_i^1, \frac{1}{R} r_i^2}(q_i^1)$ to a portion of $\mathcal{C}(S^{n-1}/\Gamma)$ hence the energy on each such sub-annulus is converging to zero (as $R\to\infty$, $i\to\infty$). But this is not sufficient to conclude the claim, as the number of dyadic sub-annuli could increase very fast with $i\to\infty$, and hence a more careful argument is required -- a delicate fact that is unfortunately ignored in some articles proving bubbling theorems. This is exactly where our energy estimate in annular neck regions from the last section comes into play. In fact, Theorem \ref{no_neck_energy} shows that the energy over the entire neck region can be estimated by the energy of the innermost and outermost dyadic sub-annuli -- for which we have just deduced that the energy converges to zero (as $R\to\infty$ and $i\to\infty$). The claim therefore follows from this theorem.
\end{proof}

Endowed with this claim, we continue the analysis of Case 2. First, we fix some $R>K^1$ sufficiently large so that
\begin{equation*}
\lim_{i\to\infty} \int_{A_{R r_i^1, \frac{1}{R} r_i^2}(q_i^1)} \abs{\Rm_{g_i}}^{n/2}_{g_i} dV_{g_i} \leq \frac{\bar{\eps} }{8N}.
\end{equation*}
Combining this with \eqref{eq.notenoughenergy}, we obtain
\begin{equation*}
\lim_{i\to\infty}\int_{A_{K^1 r_i^1, \frac{1}{R} r_i^2}(q_i^1)} \abs{\Rm_{g_i}}^{n/2}_{g_i} dV_{g_i} \leq \frac{\bar{\eps} }{8N}+\frac{\bar{\eps} }{8N} = \frac{\bar{\eps} }{4N} \leq \frac{\bar{\eps} }{4}.
\end{equation*}
Finally, combining this with \eqref{eq.secondcon} shows that for this $R$, similarly as in the case of the first bubble
\begin{equation*}
\int_{B_{h^2}(q_\infty^2,1)\setminus B_{h^2}(\hat{q}^1_\infty,\frac{1}{R})} \abs{\Rm_{h^2}}^{n/2}_{h^2} dV_{h^2} = \lim_{i\to\infty} \int_{B_{g_i}(q_i^2,r_i^2) \setminus B_{g_i}(q_i^1,\frac{1}{R}r_i^2)} \abs{\Rm_{g_i}}^{n/2}_{g_i} dV_{g_i} \geq \frac{\bar{\eps} }{4},
\end{equation*}
therefore $(V^2,h^2)$ is non-flat and thus an \emph{intermediate ALE bubble} (see Definitions \ref{ale} and \ref{bubble_types}). We might also say that $(V^2,h^2)$ is a \emph{parent} of $(V^1,h^1)$.

By smooth convergence, we have
\begin{equation}\label{eq.secondbubbleenergy}
\lim_{R\to\infty}\lim_{i\to\infty} \int_{B_{g_i}(q_i^2,Rr_i^2) \setminus B_{g_i}(q_i^1,\frac{1}{R}r_i^2)} \abs{\Rm_{g_i}}^{n/2}_{g_i} dV_{g_i} = \int_{V^2} \abs{\Rm_{h^2}}^{n/2}_{h^2} dV_{h^2}.
\end{equation}
Therefore, combining \eqref{eq.firstbubbleenergy}, \eqref{eq.secondbubbleenergy}, and Claim \ref{claim.neck}, we obtain the energy estimate
\begin{align*}
\lim_{R\to\infty} \lim_{i\to\infty} \int_{B_{g_i}(q_i^2,Rr_i^2)} \abs{\Rm_{g_i}}^{n/2}_{g_i} dV_{g_i} &= \lim_{R\to\infty} \lim_{i\to\infty} \int_{B_{g_i}(q_i^2,Rr_i^2) \setminus B_{g_i}(q_i^1,\frac{1}{R} r_i^2)} \abs{\Rm_{g_i}}^{n/2}_{g_i} dV_{g_i}\\
&\quad + \lim_{R\to\infty} \lim_{i\to\infty} \int_{B_{g_i}(q_i^1,\frac{1}{R}r_i^2) \setminus B_{g_i}(q_i^1,Rr_i^1)} \abs{\Rm_{g_i}}^{n/2}_{g_i} dV_{g_i}\\
&\quad + \lim_{R\to\infty} \lim_{i\to\infty} \int_{B_{g_i}(q_i^1,Rr_i^1)} \abs{\Rm_{g_i}}^{n/2}_{g_i} dV_{g_i}\\
&= \int_{V^2} \abs{\Rm_{h^2}}^{n/2}_{h^2} dV_{h^2} + \int_{V^1} \abs{\Rm_{h^1}}^{n/2}_{h^1} dV_{h^1}.
\end{align*}
In particular, all the energy is fully accounted for by the two bubbles detected. 

\paragraph{Further bubbles:}
We then continue to extract more bubbles and to build bubble trees, a concept which is defined as follows.
\begin{definition}[Bubble Tree]\label{bubble_tree_definition}
A \emph{bubble tree} $T$ is a tree whose vertices are ALE bubbles and whose edges are neck regions. The single ALE end of each vertex is connected by a neck region (which it meets at its smaller boundary component) to its parent and possibly further ancestors toward the \emph{root bubble} of the tree $T$, while at possibly finitely many isolated orbifold points it is connected by more necks (which it meets at their larger boundary components) to its children and possibly further descendants toward leaf bubbles of $T$. We say two bubble trees $T_1$ and $T_2$ are separable if their root bubbles are separable.
\end{definition}

We proceed inductively, assuming that we have already extracted $(\ell-1)$ point-scale sequences and the associated bubbles that will form separable bubble trees $\{T_j\}_{j\in J}$. After possibly re-labelling, we assume $\{(V^j,h^j)\}_{j\in J}$ are their separable root bubbles and we can ignore all descendants for the argument that follows. Assume further that $K^j$ are picked (as described for the first bubble above) such that for $R>K^j$ we have
\begin{equation}\label{eq.Lnotenoughenergy}
\int_{A_{K^jr_i^j, Rr_i^j}(q_i^j)} \abs{\Rm_{g_i}}^{n/2}_{g_i} dV_{g_i} \leq \frac{\bar{\eps} }{8N}, \qquad \forall j\in J.
\end{equation}
We then set
\begin{equation*}
r_i^\ell := \inf \Big\{ r>0 \, \Big\mid \int_{B_{g_i}(q,r) \setminus \bigcup_{j\in J} B_{g_i}(q_i^j,K^j r_i^j)} \abs{\Rm_{g_i}}^{n/2}_{g_i} dV_{g_i} \geq \frac{\bar{\eps} }{2} \text{ for some } B_{g_i}(q,r) \subseteq B_{g_i}(q_i,\delta) \Big\}
\end{equation*}
and let $q_i^\ell$ be points in $M_i$ such that $B_{g_i}(q_i^\ell,r_i^\ell) \subseteq B_{g_i}(q_i,\delta)$ and
\begin{equation}\label{eq.lthcon}
\int_{B_{g_i}(q_i^\ell,r_i^\ell) \setminus \bigcup_{j\in J}  B_{g_i}(q_i^j,K^j r_i^j)} \abs{\Rm_{g_i}}^{n/2}_{g_i} dV_{g_i} \geq \frac{\bar{\eps} }{2}.
\end{equation}
Note that $r^\ell_i \geq r^j_i$ for each $j \in J$ by construction. We can assume that $r^\ell_i \to 0$, otherwise there is no more curvature concentration and the process of extracting point-scale sequences stops with the $(\ell-1)$ point-scale sequences already extracted.

As for the second bubble, there are now two cases.

\paragraph{Case 1:}
For all $j \in J$
\begin{equation*}
\frac{d(q^j_i, q^\ell_i)}{r^\ell_i} \to \infty.
\end{equation*}

In this case, just as in the case of the second bubble, if we blow up using $(q_i^\ell,r_i^\ell)$, all other bubbles disappear at infinity and we obtain another \emph{leaf bubble} which is separable from all all trees $T_j$ (thus forming a new tree consisting only of one vertex).

\paragraph{Case 2:}
For some $j \in J$
\begin{equation}\label{eq.caseL2}
\frac{d(q^j_i, q^\ell_i)}{r^\ell_i} \leq M^j < \infty.
\end{equation}

In this case, denote by $\mathcal{J}\subseteq J$ the set of indices $j$ for which \eqref{eq.caseL2} holds. By assumption $\mathcal{J}\neq\emptyset$. We claim the following.

\begin{claim}\label{claim.L2}
There exists $\eta>0$ such that for each pair of indices $j\neq k$ in $\mathcal{J}$ we have
\begin{equation*}
\liminf_{i\to\infty} \frac{d(q^j_i, q^k_i)}{r^\ell_i} \geq 2\eta>0.
\end{equation*}
\end{claim}

Let us for the moment assume that the claim holds and continue. We consider the rescaled sequence $(M_i, \wt{g}_i = (r_i^\ell)^{-2}g_i, q_i^\ell)$ which by Theorem \ref{blow-up.thm} and Corollary \ref{cor.oneend} subconverges in the pointed orbifold Cheeger--Gromov sense to a complete, non-compact, Ricci-flat limit $(V^\ell, h^\ell, q_\infty^\ell)$ with bounded $L^{n/2}$ Riemannian curvature and with one end satisfying the required ALE condition. By assumption, after possibly passing to a further subsequence, for each $j\in\mathcal{J}$ the sequence $q_i^j$ converges to some $\hat{q}^j_\infty\in V^\ell$ (with $d(\hat{q}^j_\infty,q_\infty^\ell)\leq M^j$) and by Claim \ref{claim.L2} these limit points are distinct and at least distance $\eta$ away from one another. This is the crucial ingredient that allows us to proceed essentially in the exact same way as if there was only one such point. More precisely, as for the second bubble, we can conclude that these points $\hat{q}^j_\infty$ are orbifold points of $(V^\ell,h^\ell)$ and there are no other orbifold singularities. Furthermore, as in Claim \ref{claim.neck}, no energy is concentrating in the neck regions around $q_i^j$, i.e.
\begin{equation}\label{eq.Jnecks}
\lim_{R\to\infty} \lim_{i\to\infty} \int_{A_{R r_i^j, \frac{1}{R} r_i^\ell}(q_i^j)} \abs{\Rm_{g_i}}^{n/2}_{g_i} dV_{g_i} = 0, \quad \forall j\in\mathcal{J}.
\end{equation}
In particular, for $R>\max_{j\in\mathcal{J}} K^j$ sufficiently large, we obtain for every $j\in\mathcal{J}$
\begin{equation*}
\lim_{i\to\infty} \int_{A_{R r_i^j, \frac{1}{R} r_i^\ell}(q_i^j)} \abs{\Rm_{g_i}}^{n/2}_{g_i} dV_{g_i} \leq \frac{\bar{\eps} }{8N}.
\end{equation*}
Combining this with \eqref{eq.Lnotenoughenergy}--\eqref{eq.lthcon} and using the obvious estimate $\abs{\mathcal{J}} \leq N$ then implies
\begin{equation*}
\int_{B_{h^\ell}(q_\infty^\ell,1)\setminus \bigcup_{j\in \mathcal{J}} B_{h^\ell}(\hat{q}^j_\infty,\frac{1}{R})} \abs{\Rm_{h^\ell}}^{n/2}_{h^\ell} dV_{h^\ell} \geq \frac{\bar{\eps} }{2} - \abs{\mathcal{J}}\cdot \frac{\bar{\eps}}{4N} \geq \frac{\bar{\eps}}{4}.
\end{equation*}
Therefore $(V^\ell,h^\ell)$ is non-flat and thus a new parent bubble of all the bubbles $(V^j,h^j)$ with $j\in\mathcal{J}$. This means that the trees $\{T_j\}_{j\in\mathcal{J}}$ will be combined to a single tree with the new root $(V^\ell,h^\ell)$. Finally, as for the second bubble, we obtain the energy estimate
\begin{align*}
\lim_{R\to\infty} \lim_{i\to\infty} \int_{B_{g_i}(q_i^\ell,Rr_i^\ell)} \abs{\Rm_{g_i}}^{n/2}_{g_i} dV_{g_i} &= \lim_{R\to\infty} \lim_{i\to\infty} \int_{B_{g_i}(q_i^\ell,Rr_i^\ell) \setminus \bigcup_{j\in \mathcal{J}} B_{g_i}(q_i^j,\frac{1}{R} r_i^j)} \abs{\Rm_{g_i}}^{n/2}_{g_i} dV_{g_i}\\
&\quad + \sum_{j\in\mathcal{J}} \lim_{R\to\infty} \lim_{i\to\infty} \int_{B_{g_i}(q_i^j,\frac{1}{R}r_i^\ell) \setminus B_{g_i}(q_i^j,Rr_i^j))} \abs{\Rm_{g_i}}^{n/2}_{g_i} dV_{g_i}\\
&\quad + \sum_{j\in\mathcal{J}} \lim_{R\to\infty} \lim_{i\to\infty} \int_{B_{g_i}(q_i^j,Rr_i^j)} \abs{\Rm_{g_i}}^{n/2}_{g_i} dV_{g_i}\\
&= \int_{V^\ell} \abs{\Rm_{h^\ell}}^{n/2}_{h^\ell} dV_{h^\ell} + \sum_{j\in\mathcal{J}} \sum_{V^k\in T_j} \int_{V^k} \abs{\Rm_{h^k}}^{n/2}_{h^k} dV_{h^k}.
\end{align*}
In particular, all the energy is fully accounted for by $(V^\ell,h^\ell)$ and all its descendants. This uses the fact that no energy concentrates in the new neck regions according to \eqref{eq.Jnecks}, as well as the inductive assumption that the energy in each $B_{g_i}(q_i^j,Rr_i^j)$ is already fully accounted for by the bubbles in the tree $T_j$.

It remains to prove the claim.

\begin{proof}[Proof of Claim \ref{claim.L2}]
Assume towards a contradiction that there exists a non-empty subset $\mathcal{J}' \subseteq \mathcal{J}$ such that, after possibly passing to a subsequence
\begin{equation}\label{eq.57contradiction}
\lim_{i\to\infty} \frac{d(q^j_i, q^k_i)}{r^\ell_i} =0, \quad \forall j,k\in \mathcal{J}' .
\end{equation}
Then we set
\begin{equation*}
\mu_i =\min\{d(q^j_i, q^k_i) \mid  j,k\in \mathcal{J}' \} = d(q^{j_1}_i,q^{k_1}_i).
\end{equation*}
As we started with separable trees by the inductive assumption, we have $r^j_i/\mu_i \to 0$ for all $j\in \mathcal{J}'$. Therefore, as in the previous argument, the rescaled sequence $(M_i, \wt{g}_i = (\mu_i^1)^{-2}g_i, q_i^{j_1})$ subconverges in the pointed orbifold Cheeger--Gromov sense to a complete, non-compact, Ricci-flat limit $(X,h)$ with one ALE end and at most $\abs{\mathcal{J}'}$ \emph{isolated} orbifold singularities. Note also that there are at least two orbifold singularities (coming from the sequences $q_i^{j_1}$ and $q_i^{k_1}$). On the other hand, by \eqref{eq.57contradiction} we have $r_i^\ell/\mu_i \to \infty$ and therefore $(X,h)$ has energy at most $\bar{\eps}/2$ and is hence flat by Bando's gap result (Proposition \ref{prop.Bandogap}). But a flat ALE orbifold is either smooth or a flat cone with only one orbifold singularity, yielding the desired contradiction. The claim is proved.
\end{proof}

\paragraph{Termination of the process and completion of the proof of Theorem \ref{shrinker_bubbling}:}
As already noted above, the process of finding new point-scale sequences for $q\in\sQ_r$ terminates after a finite number of steps $N_q\leq N$ because for each bubble we have found disjoint regions in each $M_i$ containing at least $\bar{\eps}/4$ energy, and by assumption the energy in the ball $B_{g_i}(q_i,\delta)$ is uniformly bounded in $i$. We can therefore move on to the next orbifold point in $\sQ_r$ after a finite number of bubbles have been extracted.

By construction, Points \ref{point1} and \ref{point2} of Theorem \ref{shrinker_bubbling} obviously hold. As we made sure that the $\delta$-balls around the sequences corresponding to different orbifold points in $\sQ_r$ are disjoint, Point \ref{point4} also follows immediately. Point \ref{point3} can be seen as follows: If we take a point-scale sequence $(q_i,\varrho_i)$ as in the theorem, if it converges to a limit which is non-flat, then we must be able to detect a new region of energy concentration (disjoint to all the regions from our point-scale sequences), but this cannot happen as we have exhausted all such regions in our process. Hence, to complete the proof of the theorem, it only remains to prove the energy identity from Point \ref{point5}.

We first note that at each singular point $q\in\sQ$ there is only one tree forming. This is proved with the exact same argument a Claim \ref{claim.L2}. Assume that the tree forming at $q$ consists of bubbles $\{(V^k,h^k)\}_{k=1}^{N_q}$ and that its root bubble $(V^{N_q},h^{N_q})=:(V,h)$ is detected by a point-scale sequence $(q_i^{N_q},r_i^{N_q})=:(q_i,r_i)$. By the above construction, we already know that
\begin{equation*}
\lim_{R\to\infty} \lim_{i\to\infty} \int_{B_{g_i}(q_i,Rr_i)} \abs{\Rm_{g_i}}^{n/2}_{g_i} dV_{g_i} = \sum_{k=1}^{N_q} \int_{V^k} \abs{\Rm_{h^k}}^{n/2}_{h^k} dV_{h^k}.
\end{equation*}

There is one further neck region connecting the tree to $M_\infty$. As in Claim \ref{claim.neck}, we can show
\begin{equation*}
\lim_{R\to\infty} \lim_{i\to\infty} \int_{B_{g_i}(q_i,\frac{1}{R})\setminus B_{g_i}(q_i,Rr_i)} \abs{\Rm_{g_i}}^{n/2}_{g_i} dV_{g_i} = 0,
\end{equation*}
so that this neck also does not contribute to the total energy. Writing $\delta=1/R$, we therefore conclude that
\begin{equation*}
\lim_{\delta\to 0} \lim_{i\to\infty} \int_{B_{g_i}(q_i,\delta)} \abs{\Rm_{g_i}}^{n/2}_{g_i} dV_{g_i} = \sum_{k=1}^{N_q} \int_{V^k} \abs{\Rm_{h^k}}^{n/2}_{h^k} dV_{h^k}.
\end{equation*}
The claimed energy identity now immediately follows by repeating this for all orbifold points in $\sQ$ and combining with \eqref{eq.energywithoutdeltaballs}. Note that the condition $\sQ \cap \partial B_{g_\infty}(p_\infty,r) = \emptyset$ ensures that for sufficiently large $i$ and sufficiently small $\delta$, each $B_{g_i}(q_i,\delta)$ will be \emph{fully contained} in $B_{g_i}(q_i,r)$, avoiding potential issues with capturing ``half-bubbles''. This finishes the proof of Theorem \ref{shrinker_bubbling}.

\section{Proofs of the Corollaries from the Introduction}\label{sec_further}
We will first transform the energy identity into an identity for the Euler characteristic. This reinforces the notion that, while the formation of orbifold singularities can cause some topological degeneration, we can recover the lost topology in a quantitative and systematic way.		
\begin{proof}[Proof of Corollary \ref{cor.eulerchar}]
As noted in Anderson's work on Einstein manifolds \cite{An89_einstein}, bubbling can be excluded if the dimension $n$ is odd. In this setting we have $\mathcal{Q}_r, \mathcal{Q}^k = \emptyset$ and the result trivially holds. Therefore, we only need to consider the case when $n$ is even. The proof is rather direct and will be clear for experts; we therefore only give the full details for the easiest case $n = 4$ and then briefly point out the necessary modifications for higher dimensions. 

One of the main ingredients is the Chern-Gauss-Bonnet theorem for compact $4$-manifolds $N$ with boundary $\partial N$, namely
\begin{equation}\label{cgb_boundary}
\begin{aligned}
32\pi^2 \chi\left(N\right) &= \int_N \big(\abs{\Rm}^2 - 4\abs{\Ric}^2 + R^2\big) dV\\
&\quad + 16\int_{\partial N} \kappa_1 \kappa_2 \kappa_3 dA + 8\int_{\partial N} \big(\kappa_1 K_{23} + \kappa_2 K_{13} + \kappa_3 K_{12}\big) dA,
\end{aligned}
\end{equation}
see for example \cite{G84_euler}. Here $\kappa_a = \mathrm{II}(e_a,e_a)$ are the principal curvatures of $\partial N$ (hence $\{e_1,e_2,e_3\}$ is an orthonormal basis of $T\partial N$ diagonalising the second fundamental form), and $K_{ab} = \Rm\left(e_a,e_b,e_a,e_b\right)$ are the sectional curvatures of $N$. In particular, if $(V^k,h^k)$ is a Ricci-flat ALE orbifold with one ALE end with fundamental group $\Theta_k$ and with a finite discrete set (possibly empty) of orbifold points $\mathcal{Q}^k=\{q_\infty^{k,j}\}$ with isometry groups $\{\Gamma_{k,j}\}$, respectively, then \eqref{cgb_boundary} implies the well known formula
\begin{equation*}
\frac{1}{32\pi^2}\int_{V^k} \abs{\Rm_{h^k}}^2_{h^k} dV_{h^k} = \chi(V^k\setminus \mathcal{Q}^k)-\frac{1}{\abs{\Theta_k}}+\sum_{q_\infty^{k,j}\in\mathcal{Q}^k} \frac{1}{\abs{\Gamma_{k,j}}}
\end{equation*}

Now fix $r>0$. Take a sequence $(M_i,g_i, f_i,p_i)$ as in Theorem \ref{shrinker_bubbling}. Denote by $\mathcal{Q}_r = \mathcal{Q} \cap B_{g_\infty}(p_\infty, r)$ the orbifold points forming and fix some $q\in \mathcal{Q}_r$. Then denote by $\{(V^k,h^k)\}_{k=1}^{N_q}$ the ALE bubbles of the bubble tree $T_q$. Intermediate bubbles will have a non-empty discrete set of orbifold points $\mathcal{Q}^k=\{q_\infty^{k,j}\}$ (where the bubble tree is connected via neck regions to the children $\{(V^j,h^j)\}$) while for leaf bubbles $\mathcal{Q}^k=\emptyset$. 

By the neck theorem, as explained in the proof of Claim \ref{claim.neck}, the fundamental group at infinity $\Theta_j$ of each child bubble is the same as the orbifold group $\Gamma_{k,j}$ at $q_\infty^{k,j}$, hence these terms cancel each other when summing over all bubbles and for the entire tree $T_q$ at $q$ we find
\begin{equation}\label{cgb_onetree}
\frac{1}{32\pi^2}\sum_{k=1}^{N_q} \int_{V^k} \abs{\Rm_{h^k}}^2_{h^k} dV_{h^k} = \sum_{k = 1}^{N_q}\chi(V^k\setminus \mathcal{Q}^k)-\frac{1}{\abs{\Theta_{N_q}}}.
\end{equation}
Here $\Theta_{N_q}$ is the fundamental group at infinity of the root bubble $(V^{N_q},h^{N_q})$ of the tree $T_q$.

Similarly, we also have
\begin{equation*}
\frac{1}{32\pi^2}\bigg[\int_{B_{g_\infty}(p_\infty,r)} \big(\abs{\Rm}^2 - 4\abs{\Ric}^2 + R^2\big) dV_{g_\infty} + T(\partial B_{g_\infty}(p_\infty,r))\bigg]= \chi\big(B_{g_\infty}(p_\infty,r) \setminus \mathcal{Q}_r\big) + \sum_{q\in\mathcal{Q}_r} \frac{1}{\abs{\Gamma_q}}
\end{equation*}
where $\Gamma_q$ is the finite isometry group associated to the orbifold point $q$ and $T(\partial B_{g_\infty}(p_\infty,r))$ denotes the boundary integral in \eqref{cgb_boundary} above. Using once again the neck theorem for the neck connecting the root bubble of the tree $T_q$ to the smoothly converging body part, similarly as above, we find $\Theta_{N_q} =\Gamma_q$. Hence, by Point \ref{point5} of Theorem \ref{shrinker_bubbling} (respectively its version using the Chern-Gauss-Bonnet integrand $\abs{\Rm}^2 - 4\abs{\Ric}^2 + R^2$, which holds with identical proof\footnote{Note that in neck regions the energy vanishes and thus also the Chern-Gauss-Bonnet integral disappears. Instead in bubble regions, the energy and Chern-Gauss-Bonnet integrals agree due to Ricci-flatness}), we conclude
\begin{align*}
\lim_{i \to \infty} \chi(B_{g_i}(q_i,r)) &= \lim_{i\to\infty} \frac{1}{32\pi^2}\bigg[\int_{B_{g_i(p_i,r)}} \big(\abs{\Rm}^2 - 4\abs{\Ric}^2 + R^2\big) dV_{g_i} + T(\partial B_{g_i}(p_i,r))\bigg]\\
&=\frac{1}{32\pi^2}\bigg[\int_{B_{g_\infty}(p_\infty,r)} \big(\abs{\Rm}^2 - 4\abs{\Ric}^2 + R^2\big) dV_{g_\infty} + T(\partial B_{g_\infty}(p_\infty,r))\bigg]\\
&\quad + \frac{1}{32\pi^2}\sum_{q\in\mathcal{Q}_r}\sum_{k=1}^{N_q} \int_{V^k} \abs{\Rm_{h^k}}^2_{h^k} dV_{h^k}\\
&= \chi(B_{g_\infty}(q_\infty,r)\setminus \mathcal{Q}_r) + \sum_{q\in\mathcal{Q}_r} \frac{1}{\abs{\Gamma_q}}
 + \sum_{q \in \sQ_r} \bigg(\sum_{k = 1}^{N_q} \chi(V^k \setminus \mathcal{Q}^k)-\frac{1}{\abs{\Theta_{N_q}}}\bigg)\\
 &= \chi(B_{g_\infty}(q_\infty,r)\setminus \mathcal{Q}_r) + \sum_{q \in \sQ_r} \sum_{k = 1}^{N_q} \chi(V^k \setminus \mathcal{Q}^k).
\end{align*}
We have proved the result for $n=4$. In higher even dimensions, we can use the Chern--Gauss--Bonnet formula for a compact manifold with boundary  from Theorem $1.9.2$ in \cite{G04_asymptotic}:
\begin{equation}\label{high_dim_cgb}
\chi\left(N\right) = \int_N C_n ~\eps^I_J~ \mathcal{R}^{I,n}_{J,1} dV_g + \int_{\partial N} \sum^{n-1}_{l=0} C_{l,n}~ \eps^A_B~ \mathcal{R}^{A,2l}_{B,1}~ \mathrm{II}^{A,n-1}_{B,2l+1} dA_g.
\end{equation}
Here			
\begin{align*}
\eps^I_J &:= \eps_{i_1 \dots i_l j_1 \dots j_l} \\
\mathcal{R}^{I,t}_{J,s} &:= \Rm_{i_s i_{s+1} j_{s+1} j_s} \; \cdots\; \Rm_{i_{t-1}i_t j_t j_{t-1}}\\
\mathrm{II}^{I,t}_{J,s} &:= \mathrm{II}_{i_s j_s} \; \cdots \; \mathrm{II}_{i_t j_t}.
\end{align*}
where $\eps^I_J$ is shorthand for the Levi-Civita symbol, $\mathrm{II}_{ij}$ is the second fundamental form of $\partial N$, and $I,J$ are $(n-1)$ tuples of indices associated to an orthonormal basis $\{e_{i_1}, \ldots, e_{i_l}, e_{j_1}, \ldots, e_{j_l}\}$ of $T\partial N$ that diagonalises $\mathrm{II}_{ij}$. We also note the following:
\begin{enumerate}
\item The first term in Equation \eqref{high_dim_cgb} is the integral over a sum of products of $\frac{n}{2}$ Riemann curvature tensors and is bounded above by a multiple of the energy $E\left(r\right)$. Similar to the energy identity in Point \ref{point5} of Theorem \ref{shrinker_bubbling}, one can prove an identity for these integrals.
\item $\eps^A_B ~ \mathrm{II}^{A,n-1}_{B,1}$ is, up to a constant, the Gauss curvature of $\partial M$:
\begin{equation}\label{prin_curvs}
\eps^A_B ~ \mathrm{II}^{A,n-1}_{B,1} = \left(n-1\right) !\prod^{n-1}_{s=1} \kappa_s.
\end{equation}
where $\kappa_s$ is a principal curvature of $\partial N$. However, as in the argument in $n=4$, these terms are only needed at the boundary of $B_{g_i}(p_i),r)$ (as all ``inner'' boundary terms near orbifold points will appear twice with opposite sign and thus cancel out).
\end{enumerate}
Using Equation \eqref{high_dim_cgb} the proof is therefore the same as the $n=4$ case, up to dealing with the much more cumbersome notation.
\end{proof}


Now we will use the bubble tree construction to prove the local diffeomorphism finiteness result.

\begin{proof}[Proof of Corollary \ref{shrinker_diffeo_fin}]
Fix $r>0$. Take a sequence $(M_i,g_i, f_i,p_i)$ in $\mathcal{M}$ and assume for a contradiction that $M_i \cap B_{g_i}(p_i,r)$ have pairwise distinct diffeomorphism types. By Theorem \ref{hm_compactness_thm}, we obtain pointed orbifold convergence to an orbifold shrinker $(M_\infty,g_\infty,f_\infty,p_\infty)$. As before, let $\mathcal{Q}$ be the set of orbifold points. By possibly slightly enlarging $r$ (without relabelling it), we can assume that $\mathcal{Q} \cap \partial B_{g_\infty}(p_\infty, r) = \emptyset$. We then set $\mathcal{Q}_r = \mathcal{Q} \cap B_{g_\infty}(p_\infty, r)$. By Theorem \ref{shrinker_bubbling} we know that at each of the finitely many orbifold points $q\in \mathcal{Q}_r$ a finite number of ALE bubbles $\{(V^k,h^k)\}_{k=1}^{N_q}$ will be detected via point-scale sequences $(q_i^k,r_i^k)_{i\in\mathbb{N}}$, forming a bubble tree $T_q$. Finally assume that the last bubble, i.e. $(V^{N_q},h^{N_q})$, is the root bubble of $T_q$.

Next pick $R$ sufficiently large, so that for each bubble $(V^k,h^k)$ we have $R>K^k$ (where $K^k$ come from the bubble tree construction in Section \ref{sec_bubbling}) and $R>\eps^{-1/2}_{\mathrm{neck}}$ (where $\eps_{\mathrm{neck}}$ comes from the Neck Theorem \ref{shrinker_neck_thm}). 

Then each $B_{g_i}(p_i, r)$ can be covered by finitely many of the following regions:
\begin{enumerate}
\item \textbf{Body Regions:} These are the regions $B_{g_i}(p_i, r) \setminus \bigcup_{q\in \mathcal{Q}_r} B_{g_i}\big(q^{N_q}_i, \tfrac{1}{2R}\big)$. By the construction in Section \ref{sec_bubbling}, we have smooth convergence
\begin{equation*}
\Big(B_{g_i}(p_i, r) \setminus \bigcup_{q\in \mathcal{Q}_r} B_{g_i}\big(q^{N_q}_i, \tfrac{1}{2R}\big), g_i\Big) \to \Big(B_{g_\infty}(p_\infty, r) \setminus \bigcup_{q\in \mathcal{Q}_r} B_{g_\infty}\big(q, \tfrac{1}{2R}\big), g_\infty\Big)
\end{equation*}
and hence these regions will eventually all be diffeomorphic to each other.
\item \textbf{Bubble Regions:} Let $q$ be an orbifold point and $(V^k,h^k)$ a fixed ALE bubble of the bubble tree $T_q$. Denote by $\{(V^j,h^j)\}_{j\in J}$ its children in the bubble tree $T_q$ (with $J=\emptyset$ for a leaf bubble). Then the corresponding bubble regions are $B_{g_i}\big(q_i^k, 2R r_i^k\big) \setminus \bigcup_{j\in J} B_{g_i}\big(q^j_i, \tfrac{1}{2R}r^k_i\big)$. After rescaling with $(r_i^k)^{-2}$, these regions will smoothly converge to a region in $V^k \setminus \{q^{k,j}_\infty\}_{j \in J}$, more precisely
\begin{equation*}
\Big(B_{g_i}\big(q_i^k, 2R r_i^k\big) \setminus \bigcup_{j\in J} B_{g_i}\big(q^j_i, \tfrac{1}{2R}r^k_i\big), (r_i^k)^{-2} g_i\Big) \to \Big( B_{h^k}\big(q_\infty^k,2R\big) \setminus \bigcup_{j\in J} B_{h^k}\big(q^{k,j}_\infty,\tfrac{1}{2R}\big),h^k\Big).
\end{equation*}
This implies that the bubble regions will eventually be diffeomorphic to each other and this argument works for each of the finitely many bubbles $(V^k,h^k)$.
\item \textbf{Neck Regions:} Again, let $(V^k,h^k)$ be a fixed ALE bubble in a tree $T_q$ detected by the point-scale sequence $(q_i^k,r_i^k)$. Denote by $(V^\ell,h^\ell)$ its parent bubble, detected by $(q_i^\ell,r_i^\ell)$, if it exists. If instead $(V^k,h^k) = (V^{N_q},h^{N_q})$ is a root bubble and therefore does not have a parent, then we set $(q_i^\ell,r_i^\ell):=(q_i^k,1)$. Then $B_{g_i}(q^k_i, \tfrac{1}{R} r_i^\ell) \setminus B_{g_i}\big(q^k_i, Rr_i^k\big)$ are the corresponding neck regions. By Theorem \ref{shrinker_neck_thm}, for sufficiently large $i$ these annular regions will be diffeomorphic to an annulus on the cone $\mathcal{C}(\mathbb{S}^{n-1}/\Gamma)$ for some $\Gamma \subset O(n)$ with $\abs{\Gamma} \leq \gamma$ and thus in particular diffeomorphic to each other.
\end{enumerate}

The regions above are defined in such a way that each (annular) neck region overlaps with a bubble region on its innermost dyadic annulus and with the corresponding parent bubble region (or the body region in case of root bubbles) at its outermost dyadic annulus giving controlled regions where the diffeomorphism can be ``glued together''. In particular, after possibly passing to a subsequence, for $i$ sufficiently large $M_i \cap B_{g_i}(p_i,r)$ are all diffeomorphic to each other. This then obviously also holds true for the original (not enlarged) $r$, which is the desired contradiction. 
\end{proof}


\makeatletter
\def\@listi{%
  \itemsep=0pt
  \parsep=1pt
  \topsep=1pt}
\makeatother
{\fontsize{10}{12.5}\selectfont

\printaddress


\begin{thebibliography}{99}

\bibitem{AG90_diamgrowth}
U.~Abresch and D.~Gromoll. \textit{On Complete Manifolds with Non-Negative Ricci Curvature}. 
J. Amer. Math. Soc. Vol. 3, No. 2 (1990).

\bibitem{An89_einstein}
M.T.~Anderson. \textit{Ricci Curvature Bounds and Einstein Metrics on Compact Manifolds}. 
J. Amer. Math. Soc. Vol. 2, No. 3, pp. 455--490 (1989).

\bibitem{AC91_diffeofin}
M.T.~Anderson and J.~Cheeger. \textit{Diffeomorphism Finiteness for Manifolds with Ricci Curvature and $L^{n/2}$-Norm of Curvature Bounded}. 
Geom. Funct. Anal. Vol. 1, No. 3 (1991).

\bibitem{Ba20_nash}
R.H.~Bamler. \textit{Entropy and Heat Kernel Bounds on a Ricci Flow Background}. 
Preprint. arXiv:2008.07093v1.

\bibitem{Ba20_compactness}
R.H.~Bamler. \textit{Compactness Theory of the Space of Super Ricci Flows}. 
Preprint. arXiv:2008.09298v1.	

\bibitem{Ba20_structure}
R.H.~Bamler. \textit{Structure Theory of Non-Collapsed Limits of Ricci Flows}. 
Preprint. arXiv:2009.03243v1.

\bibitem{Ba90_bubbling}
S.~Bando. \textit{Bubbling Out of Einstein Manifolds}. 
T\^ohoku Math. J. Vol. 42, pp. 205--215 (1990). 

\bibitem{Ba90_correction}
S.~ Bando. \textit{Correction and Addition: Bubbling Out of Einstein Manifolds}. 
T\^ohoku Math. J. Vol. 42, pp. 587--588 (1990). 

\bibitem{BKN89_alecoords}
S.~Bando, A.~Kasue, and H.~Nakajima \textit{On A Construction of Coordinates at Infinity on Manifolds with Fast Curvature Decay and Maximal Volume Growth}. 
Invent. Math.Vol. 97, pp. 313--349 (1989).

\bibitem{Br97_SW}
T.~ Branson. \textit{Stein--Weiss Operators and Ellipticity}. 
J. Funct. Anal. Vol. 151, No. 2, pp. 334--383 (1997).

\bibitem{Br00_kato}
T.~Branson. \textit{Kato Constants in Riemannian Geometry}. 
Math Res. Lett. Vol. 7, pp. 245--261 (2000).

\bibitem{BC85_bubbling}
H.~Brezis and J.-M.~Coron. \textit{Convergence of Solutions of H-Systems or How to Blow Bubbles}. 
Arch. Ration. Mech. Anal. Vol. 89, pp. 21--56 (1985).

\bibitem{BS18_minsurf}
R.~Buzano and B.~Sharp. \textit{Qualitative and Quantitative Estimates for Minimal Hypersurfaces with Bounded Index and Area}. 
Trans. Amer. Math. Soc. Vol. 370, No. 6, pp. 4373--4399 (2018).

\bibitem{CGH00_kato}
D.M.J.~Calderbank, P. ~Gauduchon, and M.~Herzlich. \textit{Refined Kato Inequalities and Conformal Weights in Riemannian Geometry}.
J. Funct. Anal. Vol. 173, No. 1, pp. 214--255 (2000).

\bibitem{CS07_compactnesss}
H.-D.~Cao and N.~Sesum. \textit{A Compactness Result for K\"ahler-Ricci Solitons}. 
Adv. Math. Vol. 211, No. 2, pp. 794--818 (2007).

\bibitem{CQY07_bubbling}
S.-Y.A.~Chang, J.~Qing and P.~Yang. \textit{On a Conformal Gap and Finiteness Theorem for a Class of Four-Manifolds}. 
Geom. Funct. Anal. Vol. 17, pp. 404--434 (2007).

\bibitem{CN15_regularity}
J.~Cheeger and A.~Naber. \textit{Regularity of Einstein Manifolds and the Codimension $4$ Conjecture}. 
Ann. of Math. Vol. 182, pp. 1093--1165 (2015).

\bibitem{EMT11_typeI}
J,~Enders, R.~M\"uller and P.M.~Topping. \textit{On Type-I Singularities in Ricci Flow}. 
Comm. Anal. Geom. Vol. 19, No. 5, no. 5, pp. 905--922 (2011).

\bibitem{G84_euler}
P.~Gilkey. \textit{Invariance Theory, the Heat Equation, and the Atiyah-Singer Index Theorem}. 
Mathematics Lecture Series 11. Publish or Perish Inc. (1984).
	
\bibitem{G04_asymptotic}
P.~Gilkey. \textit{Asymptotic Formulae in Spectral Geometry}. 
Studies in Advanced Mathematics Vol. 43, Boca Raton, Florida: Chapman \& Hall/CRC (2004).

\bibitem{Ha88_surf}
R.S.~Hamilton. \textit{The Ricci Flow on Surfaces}. 
Contemp. Math., Vol. 71, pp. 237--262 (1988).

\bibitem{HM11_comp1}
R.~Haslhofer and R.~M\"uller. \textit{A Compactness Theorem for Complete Ricci Shrinkers}. 
Geom. Funct. Anal. Vol. 21 pp. 1091--1116 (2011).

\bibitem{HM15_comp2}
R.~Haslhofer and R.~M\"uller. \textit{A Note on the Compactness Theorem for $4$d Ricci Shrinkers}. 
Proc. Amer. Math. Soc. Vol. 143, No. 10, pp. 4433--4437 (2015).

\bibitem{MM15_typeI}
C.~Mantegazza and R.~M\"uller \textit{Perelman's Entropy Functional at Type I Singularities of the Ricci Flow}. 
J. Reine Angew. Math. Vol. 703, pp. 173--199 (2015).

\bibitem{MW15_cone}
O.~Munteanu and J.~Wang. \textit{Geometry of Shrinking Ricci Solitons}. 
Compositio Math. Vol. 151, pp. 2273--2300 (2015).

\bibitem{Na10_solitons}
A.~Naber. \textit{Noncompact Shrinking Four Solitons with Nonnegative Curvature}. 
J. Reine Angew. Math. Vol. 645, pp. 125--153 (2010).

\bibitem{Na88_bubbling}
H.~Nakajima. \textit{Hausdorff Convergence of Einstein $4$-Manifolds}. 
J. Fac. Sci. Univ. Tokyo. Vol. 35, pp. 411--424 (1988).

\bibitem{Pe02_entropy}
G.~Perelman. \textit{The Entropy Formula for the Ricci Flow and its Geometric Applications}. 
Preprint, arXiv:math/0211159v1.

\bibitem{SU81_bubbling}
J.~Sacks and K.~Uhlenbeck. \textit{The Existence of Minimal Immersions of $2$-Spheres}. 
Ann. of Math. Vol. 113, pp. 1--24 (1981).

\bibitem{St84_bubbling}
M.~Struwe. \textit{A Global Compactness Result For Elliptic Boundary Value Problems Involving Limiting Nonlinearities}. 
Math. Z. Vol. 187, pp. 511--517 (1984).

\bibitem{SW68}
E. ~Stein and G. ~Weiss. \textit{Generalization of the Cauchy--Riemann Equations and Representations of the Rotation Group}. Amer. J. Math. Vol. 90, No. 1, pp. 163-196 (1968).

\bibitem{Ti90_calabi}
G.~Tian. \textit{On Calabi’s Conjecture for Complex Surfaces with Positive First Chern Class}.
Invent. Math. Vol. 101, No. 1, pp. 101--172 (1990).

\bibitem{U82_remove}
K.~Uhlenbeck. \textit{Removable Singularities in Yang--Mills Fields}. 
Comm. Math. Phys. Vol. 83, No. 1, pp. 11--29 (1982).

\bibitem{We11_compactness}
B.~Weber. \textit{Convergence of Compact Ricci Solitons}. 
Int. Math. Res. Not. IMRN No. 1, pp. 96--118 (2011).

\bibitem{WW09_comparison}
G.~Wei and W.~Wylie. \textit{Comparison Geometry for the Bakry--\'Emery Ricci Tensor}. 
J. Differential Geom. Vol. 83, pp. 377--405 (2009).

\bibitem{Z06_compactness}
X.~Zhang. \textit{Compactness Theorems for Gradient Ricci Solitons}. 
J. Geom. Phys. Vol. 83, No. 12, pp. 2481--2499 (2006).

\end{thebibliography}
\end{document}